\documentclass[11pt]{article}
\usepackage{amsmath,amsthm,amssymb,amsxtra,graphicx}
\usepackage{color}
\usepackage{tikz}
\usepackage{tikz-cd}
\usepackage{mathtools}

\usetikzlibrary{positioning}
\usetikzlibrary{arrows,decorations.pathmorphing,shapes}

\def\ppp{\lower-2pt\hbox{\ \tikz\fill[black] (0,0) circle (1pt);}}
\def\plus{\raisebox{0.95pt}{\scalebox{0.5}{\ensuremath{+}}}}
\allowdisplaybreaks[4]

\newtheorem{theorem}{Theorem}[section]
\newtheorem{lemma}{Lemma}[section]

\newtheorem{corollary}{Corollary}[section]

\newcommand{\xo}{\vec{X}}
\newcommand{\yo}{\vec{Y}}

\newcommand{\sub}[2][8]{\raisebox{-#1pt}{\ensuremath{\scriptstyle #2}}}

\renewcommand{\thefootnote}{\fnsymbol{footnote}}

\date{}

\begin{document}

\title{\bf{A multilinear Zafran theorem for the measure of noncompactness of operators}}

\author{Mieczys{\l}aw~Masty{\l}o and Eduardo~B.~Silva}

\date{}

\maketitle

\noindent
\begingroup
\renewcommand{\thefootnote}{\fnsymbol{footnote}}
\footnotetext[3]{The first author was supported by the National Science Center, Poland,
project 2019/33/B/ST1/00165. The second author was supported by CNPq, Brazil, Grant 309230/2023-3.}
\endgroup

\begin{abstract}
\noindent
We establish a multilinear analogue of Zafran's interpolation theorem for operators acting on products of quasi-Banach spaces generated by the general real methods. We consider multilinear operators between quasi-Banach couples and obtain interpolation estimates expressed in terms of the fundamental functions of the underlying sequence spaces. As an application, we study the measure of noncompactness of interpolated multilinear operators in this setting and derive a corresponding interpolation estimate for this quantity. In particular, we obtain a one-sided compactness result for interpolated multilinear operators acting on spaces associated with the abstract real methods. These results recover known bilinear theorems as special cases and extend earlier results on multilinear interpolation and on the interpolation of the measure of noncompactness.
\end{abstract}

\vspace{5 mm}


\section{Introduction}

Measures of noncompactness are quantitative set functions that describe how far a bounded subset of a metric space is from being relatively compact. They play an important role in spectral theory and in the study of differential and integral equations. These notions also give rise, in a natural way, to quantitative characteristics of bounded operators through the images of bounded sets.

Determining the exact value of a measure of noncompactness is usually difficult. For applications, it is therefore useful to develop general methods for estimating such quantities for operators. This leads naturally to the problem of understanding how measures of noncompactness behave under interpolation. Even in the linear setting, such questions depend strongly on the interpolation method under consideration and on the endpoint assumptions imposed on the operator. At the same time, interpolation of nonlinear operators by abstract interpolation methods has become an active area of research, partly because of its applications to partial differential equations. Bilinear interpolation provides a notable example of this line of investigation.

In recent years, considerable attention has been devoted to the stability of compactness for compact bilinear operators under various interpolation methods. Roughly speaking, the problem is to determine whether an operator acting on products associated with Banach or quasi-Banach couples, and satisfying compactness assumptions at one or both endpoint spaces, remains compact on the corresponding interpolation spaces.

Compact bilinear operators arise naturally in harmonic analysis; see, for example, \cite{benyi1,benyi2}. In particular, results from \cite{benyi2} show that commutators of bilinear Calder\'on--Zygmund operators with symbols in the $C\!M\!O$ subspace of $B\!M\!O$ provide examples of compact bilinear operators mapping $L_p \times L_q$ into $L_r$, where $1 < p,q < \infty$ and $1/r = 1/p + 1/q \le 1$.

These results have motivated the study of interpolation properties of compact bilinear operators, a problem already considered by Calder\'on \cite{Cal} in his pioneering work on the complex interpolation method. More recently, the behavior of bilinear operators under the real interpolation method has also attracted considerable attention.

Understanding how compact bilinear operators behave under interpolation naturally leads to quantitative questions. In particular, one is led to investigate the behavior of the measure of noncompactness of a bilinear operator under interpolation. For linear operators, interpolation formulas for measures of noncompactness have been extensively studied. Analogous results are also known for two important extensions of the real method, namely the real method with a function parameter and the general real method. In \cite{MS} we developed an abstract approach that transfers results from the linear case to bilinear operators between Banach couples. The main arguments there rely on duality relations in interpolation theory.

Zafran \cite{Zafran} proved a multilinear interpolation theorem in the classical setting of Lions--Peetre spaces. One of the main aims of the present article is to establish a multilinear analogue of this theorem for spaces generated by the general real methods. This interpolation theorem is one of the main tools of the article and plays a central role in our study of the measure of noncompactness of multilinear operators in this framework.

In this article, we study multilinear interpolation and the behavior of the measure of noncompactness for multilinear operators acting between quasi-Banach spaces associated with the general real method. Our approach is direct and is based on properties of vector-valued sequence spaces arising in the construction of this method. The main idea is to decompose the operator into smaller parts by means of families of projections on sequence spaces. For each part, we estimate the measure of noncompactness, in a way related to the approach in \cite{BC} for bilinear operators acting between quasi-Banach spaces. As a consequence of the main result, we also obtain a one-sided compactness theorem for interpolated multilinear operators in this abstract setting.

These results recover several known multilinear interpolation theorems for real methods as special cases, including the classical theorem of Zafran for Lions--Peetre spaces. In the classical setting, we also show that the corresponding condition on the interpolation parameters is sharp. Moreover, the results extend earlier work on the interpolation of the measure of noncompactness to the multilinear framework. For further information on multilinear interpolation, we refer to the works of Calder\'on \cite{Cal}, Lions--Peetre \cite{LP}, and Zafran \cite{Zafran}. For compact bilinear operators and the interpolation of measures of noncompactness, we refer to the works of Cobos and collaborators \cite{BC,CLM1}, and to the references therein.

After the completion of an earlier version of the present work, we became aware of the recent paper of Cobos, Fern\'andez-Cabrera and K\"uhn \cite{CFCK}, in which the authors investigate compact multilinear operators under the classical Lions--Peetre real interpolation method between quasi-Banach spaces and obtain, as an application, a reinforced multilinear Marcinkiewicz theorem. The results of the present article are formulated in the setting of the general real method and provide quantitative estimates for measures of noncompactness; in particular, they yield a one-sided compactness theorem for interpolated multilinear operators.

\medskip

\noindent\textit{Note on the submission history.}
An earlier version of this manuscript was submitted to \textit{Constructive Approximation} on March $26, 2025$. On March $16, 2026$, we were informed by the journal that, owing to an editorial oversight, no further editorial action had been taken on the manuscript. We subsequently withdrew the manuscript in May 2026.

\section{Notation and preliminary results}

We will consider multilinear mappings acting on quasi-Banach spaces. Recall that a quasi-normed space is a vector space $X$ endowed with a map $\|\cdot\| \colon X \to [0,\infty)$, also denoted by $\|\cdot\|_X$, such that

\begin{itemize}
\item [{\rm(i)}] $\|x\| = 0$ if and only if $x = 0$;
\item [{\rm(ii)}] $\|\lambda x\| = |\lambda| \|x\|$ for all $\lambda \in \mathbb{K}$ and $x \in X$;
\item [{\rm(iii)}] there exists a constant $\kappa \ge 1$ such that
\[
\|x+y\| \le \kappa \big(\|x\|+\|y\|\big)
\]
for all $x,y \in X$.
\end{itemize}

The smallest constant $\kappa$ for which {\rm(iii)} holds is called the modulus of concavity of the quasi-norm and is denoted by $c_X$. If $c_X=1$, then $\|\cdot\|$ is a norm.

\noindent
A quasi-norm $\|\cdot\|$ on $X$ is called a $p$-norm ($0 < p < 1$) if it is $p$-subadditive, that is, if it satisfies
\[
\|x + y\|^p \leq \|x\|^p + \|y\|^p, \quad\, x, y \in X\,.
\]
A $p$-subadditive quasi-norm $\|\cdot\|_X$ induces a metric topology on $X$. In fact, a metric can be defined by $d_p(x, y) = \|x - y\|^p$ for all $x, y \in X$.
The space $X$ is called a quasi-Banach space if $X$ is complete with respect to this metric.

Recall that the classical Aoki-Rolewicz theorem (see \cite{Rolewicz}) states that for every quasi-normed space $(X, \|\cdot\|)$ with modulus of concavity $\kappa > 1$, there exists 
$0 < p < 1$ and an equivalent $p$-norm $\|\cdot\|_*$ such that
\[
\|x\|_{*} \leq \|x\| \leq 2 \kappa \|x\|_{*}, \quad\, x \in X\,.
\]
Note that the $p$-norm $\|\cdot\|_{*}$ is given by
\[
\|x\|_{*} = \inf \bigg\{\Big( \sum_{k=1}^n \|x_k\|^{p} \Big)^{1/p} : \, x = \sum_{k=1}^n x_k, x_k \in X, n \in \mathbb{N} \bigg\}\,.
\]
For further information on the theory of quasi-Banach spaces, we refer to the books by Rolewicz \cite{Rolewicz} and Kalton-Peck-Roberts \cite{KPR}.

As usual, an operator $T \colon X \to Y$ between quasi-normed spaces is a continuous linear mapping. If $X_1, \ldots, X_n$ are quasi-normed spaces, 
then the product space $X_1 \times \cdots \times X_n$ is equipped with the standard quasi-norm
\[
\|(x_1, \ldots, x_n)\| := \max\{\|x_1\|_{X_1}, \ldots, \|x_n\|_{X_n}\}, \quad\, (x_1, \ldots, x_n) \in X_1 \times \cdots \times X_{n}\,.
\]
An $n$-linear mapping $T \colon X_1 \times \cdots \times X_n \to Y$ between quasi-normed spaces is called bounded if
\[
\|T\| := \sup\big\{\|T(x_1, \ldots, x_n)\|_{Y} : \, (x_1, \ldots, x_n) \in B_{X_1} \times \cdots \times B_{X_n}\big\} < \infty\,,
\]
where $B_X$ (or $B_{(X, \|\cdot\|_X)}$) denotes the closed unit ball in the quasi-normed space $X := (X, \|\cdot\|_X)$, that is,
$B_X := \{x \in X : \, \|x\|_X \leq 1\}$.  

The space of all bounded multilinear mappings $T \colon X_1 \times \cdots \times X_n \to Y$ is denoted by $L(X_1 \times \cdots \times X_n; Y)$. An $n$-linear mapping 
$T \colon X_1 \times \cdots \times X_n \to Y$ is said to be compact if $T(B_{X_1} \times \cdots \times B_{X_n})$ is precompact.

Let $\xo = (X^0, X^1)$ be a quasi-Banach couple (or $p$-normed quasi-Banach couple), that is, two ($p$-normed) quasi-Banach spaces $X^0$ and $X^1$ which are continuously 
embedded in the same Hausdorff topological vector space. For $t > 0$, Peetre’s $K$- and $J$-functionals are defined by
\[
K(t, x) = K(t, x; X^0, X^1) := \inf\big\{\|x_0\|_{X^0} + t \|x_1\|_{X^1} : \, x = x_0 + x_1\big\}, \quad\, x \in X^0 + X^1
\]
and
\[
J(t, x) = J(t, x; X^0, X^1) := \max\big(\|x\|_{X^0}, t \|x\|_{X^1}\big), \quad\, x \in X^0 \cap X^1\,.
\]

The functionals $K(t, \cdot)$ and $J(t, \cdot)$ are quasi-norms on $X^0 + X^1$ and $X^0 \cap X^1$, respectively, and the quasi-triangle 
inequality is satisfied with the constant $c = \max\{c_{X^0}, c_{X^1}\}$.

If $\|\cdot\|_{X^0}$ and $\|\cdot\|_{X^1}$ are $p$-norms, then $J(t, \cdot)$ is also a $p$-norm on $X^0 \cap X^1$, and the functional
\[
K_p(t,x) = \inf\big\{\big(\|x_0\|^p_{X^0} + t^p \|x_1\|^p_{X^1}\big)^{1/p} : \, x = x_0 + x_1 \,, x_j \in X^j \,, j = 0,1\big\},
\]
is a $p$-norm on $X^0 + X^1$, which is equivalent to $K(t, \cdot)$. Namely,
\[
K(t,x) \leq K_p(t,x) \leq 2^{1/p} K(t,x), \quad\, x \in X^0 + X^1\,.
\]

Let $E$ be a quasi-Banach sequence lattice. We say that $E$ is $K$-non-trivial if $(\min(1, 2^m))_{m \in \mathbb{Z}} \in E$. The lattice $E$ is said 
to be $(p, J)$-non-trivial, $0 < p \leq 1$, if
\[
\sup \bigg\{ \Big( \sum_{m=-\infty}^\infty (\min(1,2^{-m}) |\xi_m|)^p \Big)^{1/p} : \|(\xi_m)\|_E \leq 1 \bigg\} < \infty\,.
\]
Note that if $E$ is $(p, J)$-non-trivial then $E$ is also $(r, J)$-non-trivial for any $p \leq r \leq 1$. 

Let $\omega(\mathbb{Z})$ be the space of all real sequences indexed by $\mathbb{Z}$, and let $E \subset \omega(\mathbb{Z})$ be a $K$-non-trivial quasi-Banach sequence 
lattice on $\mathbb{Z}$. Additionally, let $\xo = (X^0, X^1)$ denote a quasi-Banach couple. The general real interpolation space realized by means of the $K$-functional 
$K_E(\xo)$ (denoted also by $\xo_{E, K}$) consists of all $x \in X^0 + X^1$ such that $(K(2^m, x))_{m \in \mathbb{Z}} \in E$. The quasi-norm on $\xo_{E, K}$ is given by 
$\|x\|_{K_E(\xo)} := \|(K(2^m, x))_{m \in \mathbb{Z}}\|_E$.

If $E$ is a $(p, J)$-non-trivial quasi-Banach sequence lattice and $\xo = (X^0,X^1)$ is a $p$-normed quasi-Banach couple, the general real interpolation space 
realized by means of the $J$-functional $J_E(\xo)$ (denoted also by $\xo_{E,J}$) is defined as the set of all $x \in X^0 + X^1$ which admits a representation as 
$x = \sum_{m=-\infty}^\infty u_m$, with convergence in $X^0 + X^1$, where $(u_m) \subset X^0 \cap X^1$, and $(J(2^m, u_m))_{m \in \mathbb{Z}} \in E$. The quasi-norm on $\xo_{E,J}$ 
is given by  
\[
\|x\|_{J_E(\xo)} = \inf\bigg\{ \|(J(2^m, u_m))_{m \in \mathbb{Z}}\|_{E} : \, x = \sum_{m=-\infty}^\infty u_m \bigg\}\,.
\]
We have that $X^0 \cap X^1 \hookrightarrow K_E(X^0,X^1) \hookrightarrow J_E(X^0,X^1) \hookrightarrow X^0 + X^1$, where $\hookrightarrow$ means continuous inclusion. 
The embedding $J_E(X^0,X^1) \hookrightarrow K_E(X^0,X^1)$ holds since the Calderón transform given by the formula
\[
\Omega_p((\xi_m)) = \bigg( \Big( \sum_{k=-\infty}^\infty (\min(1, 2^{m-k}) |\xi_k|)^p \Big)^{1/p}\bigg)_{m \in \mathbb{Z}}
\]
is bounded on $E$. Hence, it holds that if 
\begin{equation}
E \text{ is } K\text{-non-trivial, } (p, J)\text{-non-trivial, and } \Omega_p \text{ is bounded on } E\,,
\end{equation}
then, for any $p$-normed quasi-Banach couple $\xo = (X^0, X^1)$, we have that $\xo_{E, K} = \xo_{E,J}$ up to equivalence of quasi-norms. In this case, for simplicity 
of notation, we write $\xo_E$ (or $(X^0, X^1)_E$), and use $\|\cdot\|_{\xo_E}$ to denote either of the two quasi-norms. 

For $0 < q \leq \infty$, $\ell_q$ is the usual sequence space of scalar sequences indexed by $\mathbb{Z}$. Let $(\lambda_m) \in \omega(\mathbb{Z})$ be a positive 
sequence and let $(W_m)$ be a sequence of quasi-Banach spaces with the same constant $c \geq 1$ in the quasi-triangle inequality for all $W_m$. We define
\[
\ell_q(\lambda_m W_m) = \{w = (w_m) : w_m \in W_m \mbox{ and } (\lambda_m \|w_m\|) \in \ell_q \}\,.
\]
The quasi-norm in $\ell_q(\lambda_m W_m)$ is given by $\|w\|_{\ell_q(\lambda_m W_m)} = \|(\lambda_m \|w_m\|_{W_m})\|_{\ell_q}$. We define the space $E(\lambda_m W_m)$ similarly. 
If $W_m = \mathbb{R}$ or $W_m = \mathbb{C}$ for all $m \in \mathbb{Z}$, then we simply write $\ell_q(\lambda_m)$.

For any $k \in \mathbb{Z}$, the shift operator $\tau_k \colon \omega(\mathbb{Z}) \to \omega(\mathbb{Z})$ is defined by 
\[
\tau_k(\xi) = (\xi_{m + k})_{m \in \mathbb{Z}}, \quad\, \xi = (\xi_m) \in \omega(\mathbb{Z})\,.
\]
Throughout the paper, we consider quasi-Banach lattices on $\mathbb{Z}$ such that $\tau_k$ is bounded on $E$ for all $k \in \mathbb{Z}$.
Clearly, for each $m, n \in \mathbb{Z}$, we have
\[
\|\tau_{n+m}\|_{E \to E} \leq \|\tau_{n}\|_{E \to E} \|\tau_m\|_{E \to E}.
\]
Additionally, we assume that 
\begin{equation} \label{cite1}
\lim_{n \to \infty} 2^{-n} \|\tau_n\|_{E \to E} = \lim_{n \to \infty} \|\tau_{-n}\|_{E \to E} = 0 \,.
\end{equation}
Following \cite{Szw}, we define $\varphi_E\colon (0, \infty) \to (0, \infty)$ by the formula
\[
\varphi_E(t) := \big\|\tau_{[\log_2 t]}\big\|_{E\to E}, \quad\, t > 0\,,
\]
where the logarithm is taken in base $2$ and $[\,\cdot\,]$ is the greatest integer function.

If we let $\gamma_1 = \max(1, \|\tau_1\|_{E \to E})$, $\gamma_2 = \sup\{\varphi_E(t): \, 0 < t \leq 1\} = \sup\{ \|\tau_{-n}\|_{E \to E}: \, n \geq 0\}$, and 
$\gamma_3 = \sup \{\varphi_E(t)/t: \, 1 \leq t < \infty\} = \sup \{2^{-n} \|\tau_n\|_{E \to E} : n \geq 0\}$, we have that the following easily verified properties 
hold for the function $\varphi_E$, which we will use without further reference (for more details, see \cite{Szw}):
\begin{itemize}
\item[$\ppp$] $\varphi_E(t) = o(\max(1,t))$ as $t \to 0\plus$ or $t \to \infty$.
\item[$\ppp$] $\varphi_E(s t) \leq \gamma_1 \varphi_E(s) \varphi_E(t)$ and $\varphi_E(2^m s) \leq \varphi_E(2^m) \varphi_E(s)$,\, $s, t > 0$, $m \in \mathbb{Z}$.
\item[$\ppp$] $\varphi_E(s) \leq \gamma_1 \gamma_2 \varphi_E(t)$ and $\frac{\varphi_E(t)}{t} \leq \gamma_1 \gamma_3 \frac{\varphi_E(s)}{s}$,\, $0 < s < t$.
\end{itemize}

For fundamental notation in interpolation theory, we refer to \cite{BL} and \cite{BK}. Given two quasi-Banach couples $(X^0, X^1)$ and $(Y^0, Y^1)$, the notation 
$T \colon (X^0, X^1) \to (Y^0, Y^1)$ indicates that $T \colon X^0 + X^1 \to Y^0 + Y^1$ is a linear operator such that the restrictions of $T$ to each of the spaces 
$X^j$ (for $j = 0, 1$) are bounded operators from $X^j$ to $Y^j$.

\section{Measures of noncompactness and convolution}

Let us recall the definition of the \emph{measure of noncompactness}, which was first introduced by Kuratowski (1930). Let $(M, \rho)$ be a complete metric space, 
and let $\mathcal{B}$ denote the collection of nonempty, bounded subsets of $M$. The Kuratowski \emph{measure of noncompactness} 
$\alpha_{M} \colon \mathcal{B} \to \mathbb{R}_{+}$ (denoted $\alpha$ for brevity) is defined for any $B \in \mathcal{B}$ as  
\[
\alpha_{M}(B) := \inf \Big\{ \delta > 0 \mid B \subset \bigcup_{i=1}^n A_i, \; \text{where} \; \text{diam}(A_i) 
\leq \delta, \; 1 \leq i \leq n, \; n \in \mathbb{N} \Big\},
\]
where $\text{diam}(A) := \sup \left\{ \rho(x, y) \mid x, y \in A \right\}$ denotes the diameter of a set $A \in \mathcal{B}$.  
It is straightforward to verify that  
\[
\alpha(B) = \alpha(\bar{B}), \quad \text{and} \quad \alpha(B) = 0 \quad \text{if and only if} \quad B \text{ is relatively compact}.
\]  
Clearly, $\alpha$ is a non-decreasing set function, meaning that if $A \subset B \in \mathcal{B}$, then  
\[
\alpha(A) \leq \alpha(B).
\]  

A commonly used measure is the \emph{Hausdorff (or ball) measure of noncompactness}. This function,  
$\chi_{M} \colon \mathcal{B} \to \mathbb{R}_{+}$ (denoted $\chi$ for brevity), is defined for $B \in \mathcal{B}$ as  
\[
\chi_{M}(B) := \inf \Big\{ r > 0 \mid B \subset \bigcup_{i=1}^{n} B(x_i, r), \; \text{where} \; x_i \in M, \; 1 \leq i \leq n, \; n \in \mathbb{N} \Big\},
\]
where $B(x, r) := \left\{ y \in M: \, \rho(x, y) \leq r \right\}$ denotes the closed ball centered at $x$ with radius $r > 0$.  

We list below some of the most important properties of the measure of noncompactness $\psi$, where $\psi := \alpha\sub[4]{X}$ or $\psi := \chi\sub[4]{X}$, 
in the case of Banach spaces. Let $A$ and $B$ be bounded subsets of a Banach space $X$. Then, we have:

\begin{itemize}
\item[{\rm (i)}] $\widetilde{\beta}(T \colon X_1 \times \cdots \times X_n \to Y) \leq \|T\|$;
\item[{\rm (ii)}] $\widetilde{\beta}(T) \leq \beta(T) \leq 2 \widetilde{\beta}(T)$;
\item[{\rm (iii)}] $T$ is compact if and only if $\beta(T) = 0$;
\item[{\rm (iv)}] $\widetilde{\beta}(T) \leq 2 c_Z \widetilde{\beta}(R T)$ for any metric injection $R \colon Y \to Z$, that is, $\|Ry\|_Z = \|y\|_Y$ for all $y \in Y$, where $Z$ is a quasi-Banach space;
\item[{\rm (v)}] If $Z_1, \ldots, Z_n$ are quasi-Banach spaces and $P_1, \ldots, P_n$ are bounded linear operators, $P_i \in L(Z_i, X_i)$, then 
\[
\widetilde{\beta}(T(P_1, \ldots, P_n)) \leq \|P_1\|_{L(Z_1, X_1)} \cdots \|P_n\|_{L(Z_n, X_n)} \widetilde{\beta}(T).
\]
  
Moreover, if for any $x_i \in X_i$, with $\|x_i\|_{X_i} < 1$, $i = 1, \ldots, n$, there exist $z_i \in Z_i$ with $\|z_i\|_{Z_i} < 1$, and $P_i(z_i) = x_i$, then
\[
\widetilde{\beta}(T) \leq \widetilde{\beta}(T(P_1, \ldots, P_n)).
\]
  
\item[{\rm (vi)}] $\widetilde{\beta}(S + T) \leq \widetilde{\beta}(S) + \widetilde{\beta}(T)$.
\end{itemize}

Now, we introduce the multilinear version of convolution. If $\xi_i = (\xi_{i,m})_{m \in \mathbb{Z}}$ are sequences of 
non-negative scalars, $1 \leq i \leq n$, their convolution is given by:
\[
\xi_1 \ast \cdots \ast \xi_n = \bigg(\sum_{k_{n-1}=-\infty}^\infty \sum_{k_{n-2}=-\infty}^\infty 
\cdots \sum_{k_1=-\infty}^\infty \xi_{1,k_{n-1}} \xi_{2,k_{n-2}} \cdots \xi_{n-1,k_1} \xi_{n,m - (k_1 + k_2 + \cdots + k_{n-1})}\bigg)_{m \in \mathbb{Z}}\,.
\]

\medskip

\begin{theorem}\label{conv}
Let $n\ge 2$. Let $\xo_i=(X_i^0,X_i^1)$, $1\le i\le n$, be a $p_i$-normed quasi-Banach couple, and let $\yo=(Y^0,Y^1)$ be an $r$-normed quasi-Banach couple, where $0<p_i,r\le 1$. Suppose that $E_1$ and $E$ are $K$-non-trivial quasi-Banach sequence lattices, and that $E_2,\ldots,E_n$ are, respectively, $(p_2,J),\ldots,(p_n,J)$-non-trivial quasi-Banach sequence lattices satisfying condition \eqref{cite1}.

Furthermore, suppose that there exists a constant $\gamma > 0$ such that for all sequences $\xi_1 \in E_1, \ldots, \xi_n \in E_n$, the following estimate holds:
\begin{equation}\label{eq1A}
\big\|\big(|\xi_1|^r \ast \cdots \ast |\xi_n|^r\big)^{1/r}\big\|_E
\leq \gamma \|\xi_1\|_{E_1} \cdots \|\xi_n\|_{E_n}\,.
\end{equation}

Let $T \colon \xo_1 \times \cdots \times \xo_n \to \yo$ be an $n$-linear mapping, and let
\[
\|T\|_j := \|T\|_{L(X^j_1,\ldots,X^j_n;Y^j)}, \quad\, j=0,1\,.
\]
Then the restriction of $T$ is bounded from $K_{E_1}(\xo_1) \times J_{E_2}(\xo_2) \times \cdots \times J_{E_n}(\xo_n)$ to $K_E(\yo)$, and its norm satisfies
\[
\|T\| \leq C \|T\|_0 \varphi_{E_2}\bigg(\bigg(\frac{\|T\|_1}{\|T\|_0}\bigg)^{1/(n-1)}\bigg) \cdots \varphi_{E_n}\bigg(\bigg(\frac{\|T\|_1}{\|T\|_0}\bigg)^{1/(n-1)}\bigg),
\]
provided that $\|T\|_0>0$ and $\|T\|_1>0$, where $C$ is a constant independent of $T$. If $\|T\|_0=0$ or $\|T\|_1=0$, then $\|T\|=0$.
\end{theorem}

\begin{proof}
Let $\sigma_j > \|T\|_j$, $j = 0,1$. Choose $\nu \in \mathbb Z$ such that
\[
2^{(n-1)\nu} \le \frac{\sigma_1}{\sigma_0} < 2^{(n-1)(\nu+1)}\,.
\]

Let $x \in K_{E_1}(\xo_1)$ and $u_i \in X_i^0 \cap X_i^1$, $2 \le i \le n$.
Fix $m,k_2,\ldots,k_n \in \mathbb Z$. If $x = x_0 + x_1$, where $x_j \in X_1^j$,
$j=0,1$, then
\[
T(x,u_2,\ldots,u_n)
=
T(x_0,u_2,\ldots,u_n)
+
T(x_1,u_2,\ldots,u_n)\,.
\]
Then we obtain
\begin{eqnarray*}
\lefteqn{K(2^m, T(x,u_2,\ldots,u_n))} \nonumber \\
&=& K(2^m,T(x_0,u_2,\ldots,u_n) + T(x_1,u_2,\ldots,u_n)) \nonumber \\
&\leq& \|T(x_0,u_2,\ldots,u_n)\|_{Y^0} + 2^m \|T(x_1,u_2,\ldots,u_n)\|_{Y^1} \nonumber \\
&\leq& \|T\|_0 \|x_0\|_{X_1^0} \|u_2\|_{X_2^0} \cdots \|u_n\|_{X_n^0}  + 2^m \|T\|_1 \|x_1\|_{X_1^1} \|u_2\|_{X_2^1} \cdots \|u_n\|_{X_n^1} \nonumber \\
&\leq& \sigma_0 \|x_0\|_{X_1^0} \|u_2\|_{X_2^0} \cdots \|u_n\|_{X_n^0} + 2^m \sigma_1 \|x_1\|_{X_1^1} \|u_2\|_{X_2^1} \cdots \|u_n\|_{X_n^1} \nonumber \\
&=& \sigma_0 \|x_0\|_{X_1^0} \|u_2\|_{X_2^0} \cdots \|u_n\|_{X_n^0} \nonumber \\
&& \quad\, + 2^m 2^{-((k_2+\nu)+\cdots+(k_n+\nu))} 2^{k_2+\nu} \cdots 2^{k_n+\nu} \sigma_1
\|x_1\|_{X_1^1} \|u_2\|_{X_2^1} \cdots \|u_n\|_{X_n^1} \nonumber \\
&=& \sigma_0 \|x_0\|_{X_1^0} \|u_2\|_{X_2^0} \cdots \|u_n\|_{X_n^0} \nonumber \\
&& \quad\, + 2^{m-k_2-\cdots-k_n-(n-1)\nu} 2^{k_2+\nu} \cdots 2^{k_n+\nu} \sigma_1
\|x_1\|_{X_1^1} \|u_2\|_{X_2^1} \cdots \|u_n\|_{X_n^1} \nonumber \\
&\leq& \max(\sigma_0,2^{-(n-1)\nu}\sigma_1)\big(
\|x_0\|_{X_1^0} \|u_2\|_{X_2^0} \cdots \|u_n\|_{X_n^0} \nonumber \\
&& \quad\, + 2^{m-k_2-\cdots-k_n} \|x_1\|_{X_1^1}
2^{k_2+\nu}\|u_2\|_{X_2^1} \cdots 2^{k_n+\nu}\|u_n\|_{X_n^1}
\big) \nonumber \\
&\leq& \max(\sigma_0,2^{-(n-1)\nu}\sigma_1)\bigg(
\|x_0\|_{X_1^0} \max(\|u_2\|_{X_2^0},2^{k_2+\nu}\|u_2\|_{X_2^1}) \cdots \nonumber \\
&& \qquad\qquad \cdots \max(\|u_n\|_{X_n^0},2^{k_n+\nu}\|u_n\|_{X_n^1}) \nonumber \\
&& \quad\, + 2^{m-k_2-\cdots-k_n} \|x_1\|_{X_1^1}
\max(\|u_2\|_{X_2^0},2^{k_2+\nu}\|u_2\|_{X_2^1}) \cdots \nonumber \\
&& \qquad\qquad \cdots \max(\|u_n\|_{X_n^0},2^{k_n+\nu}\|u_n\|_{X_n^1})
\bigg) \nonumber \\
&=& \max(\sigma_0,2^{-(n-1)\nu}\sigma_1)
J(2^{k_2+\nu},u_2) \cdots J(2^{k_n+\nu},u_n) \nonumber \times \big(\|x_0\|_{X_1^0} + 2^{m-k_2-\cdots-k_n}\|x_1\|_{X_1^1}\big)\,.
\end{eqnarray*}
Taking the infimum over all decompositions $x = x_0 + x_1$, we get
\begin{eqnarray*}
\lefteqn{K(2^m, T(x,u_2,\ldots,u_n))} \nonumber \\
&\leq& \max(\sigma_0,2^{-(n-1)\nu}\sigma_1)
J(2^{k_2+\nu},u_2) \cdots J(2^{k_n+\nu},u_n) \nonumber \\
&& \quad\, \times K(2^{m-k_2-\cdots-k_n},x)\,.
\end{eqnarray*}
Since
\[
\frac{\sigma_1}{\sigma_0} < 2^{(n-1)(\nu+1)}\,,
\]
we have
\[
\max(\sigma_0,2^{-(n-1)\nu}\sigma_1)
\leq
\max\big(\sigma_0,2^{-(n-1)\nu}2^{(n-1)(\nu+1)}\sigma_0\big)
=
2^{n-1}\sigma_0\,,
\]
and hence
\begin{equation}\label{eq2A}
K(2^m, T(x,u_2,\ldots,u_n))
\leq
2^{n-1}\sigma_0 \, K(2^{m-k_2-\cdots-k_n},x)
J(2^{k_2+\nu},u_2) \cdots J(2^{k_n+\nu},u_n)\,.
\end{equation}

Now let $x_i \in J_{E_i}(\xo_i)$, $2 \le i \le n$. Consider decompositions
\[
x_i = \sum_{k_i=-\infty}^{\infty} u^i_{k_i}
\quad\, \text{(convergence in $X_i^0 + X_i^1$)}\,,
\]
where $u^i_{k_i} \in X_i^0 \cap X_i^1$. Since the above series converges absolutely in $X_i^0 + X_i^1$, it follows that for each integer $\nu$
\[
x_i = \sum_{k_i=-\infty}^{\infty} u^i_{k_i+\nu}
\quad\, \text{(convergence in $X_i^0 + X_i^1$)}\,.
\]

Recall that $K_r$ is an $r$-norm on $Y^0 + Y^1$, which is equivalent to $K$, so
\[
K(t,\cdot) \leq K_r(t,\cdot) \leq C_1 K(t,\cdot)\,.
\]
In consequence,
\begin{eqnarray}\label{eq3A}
\lefteqn{K(2^m,T(x,x_2,\ldots,x_n))
=K\Big(2^m,
\sum_{k_2=-\infty}^{\infty} \cdots \sum_{k_n=-\infty}^{\infty}
T(x,u^2_{k_2+\nu},\ldots,u^n_{k_n+\nu})\Big)} \nonumber\\
&\leq& K_r\Big(2^m,
\sum_{k_2=-\infty}^{\infty} \cdots \sum_{k_n=-\infty}^{\infty}
T(x,u^2_{k_2+\nu},\ldots,u^n_{k_n+\nu})\Big) \nonumber\\
&\leq& \Big(\sum_{k_2=-\infty}^{\infty} \cdots
\sum_{k_n=-\infty}^{\infty}
K_r(2^m,T(x,u^2_{k_2+\nu},\ldots,u^n_{k_n+\nu}))^r\Big)^{1/r}
\nonumber\\
&\leq& C_1\Big(\sum_{k_2=-\infty}^{\infty} \cdots
\sum_{k_n=-\infty}^{\infty}
K(2^m,T(x,u^2_{k_2+\nu},\ldots,u^n_{k_n+\nu}))^r\Big)^{1/r}.
\end{eqnarray}

Combining \eqref{eq1A}, \eqref{eq2A} and \eqref{eq3A}, we get
\begin{eqnarray*}
\lefteqn{\|T(x,x_2,\ldots,x_n)\|_{K_E(\yo)}
=
\big\|(K(2^m,T(x,x_2,\ldots,x_n)))_m\big\|_E} \nonumber \\
&\leq& C_1 \Big\|\Big(\sum_{k_2=-\infty}^{\infty} \cdots \sum_{k_n=-\infty}^{\infty}
K(2^m,T(x,u^2_{k_2+\nu},\ldots,u^n_{k_n+\nu}))^r\Big)^{1/r}\Big\|_E \nonumber \\
&\leq& 2^{n-1} C_1 \sigma_0
\Big\|\Big(\sum_{k_2=-\infty}^{\infty} \cdots \sum_{k_n=-\infty}^{\infty}
K(2^{m-k_2-\cdots-k_n},x)^r \nonumber \\
&& \qquad\qquad \times J(2^{k_2+\nu},u^2_{k_2+\nu})^r \cdots
J(2^{k_n+\nu},u^n_{k_n+\nu})^r \Big)^{1/r}\Big\|_E \nonumber \\
&\leq^*& 2^{n-1} C_1 \sigma_0
\Big\|\Big(\sum_{j_2=-\infty}^{\infty} \cdots \sum_{j_n=-\infty}^{\infty}
K(2^{j_2},x)^r J(2^{j_3+\nu},u^2_{j_3+\nu})^r \nonumber \\
&& \qquad\qquad \cdots J(2^{j_n+\nu},u^{n-1}_{j_n+\nu})^r
J(2^{m-(j_2+\cdots+j_n)+\nu},u^n_{m-(j_2+\cdots+j_n)+\nu})^r
\Big)^{1/r}\Big\|_E \nonumber \\
&\leq& 2^{n-1} C_1 \gamma \sigma_0
\|(K(2^m,x))_m\|_{E_1}
\|(J(2^{m+\nu},u^2_{m+\nu}))_m\|_{E_2} \cdots
\|(J(2^{m+\nu},u^n_{m+\nu}))_m\|_{E_n} \nonumber \\
&=& 2^{n-1} C_1 \gamma \sigma_0 \|x\|_{K_{E_1}(\xo_1)}
\|\tau_\nu((J(2^m,u^2_m))_m)\|_{E_2} \cdots
\|\tau_\nu((J(2^m,u^n_m))_m)\|_{E_n} \nonumber \\
&\leq& 2^{n-1} C_1 \gamma \sigma_0 \|x\|_{K_{E_1}(\xo_1)}
\|\tau_\nu\|_{E_2 \to E_2} \|(J(2^m,u^2_m))_m\|_{E_2} \cdots \nonumber \\
&& \qquad\qquad \cdots
\|\tau_\nu\|_{E_n \to E_n} \|(J(2^m,u^n_m))_m\|_{E_n}\,.
\end{eqnarray*}
In the line marked by $\leq^*$, we make the change of variables
\[
j_2 = m-(k_2+\cdots+k_n), \quad j_3 = k_2, \quad j_4 = k_3, \quad\, \dots, \quad\, j_n = k_{n-1}\,
\]
and apply the convolution inequality (3) to the sequences
\[
\xi_1 = (K(2^m,x))_{m\in\mathbb Z}, \quad
\xi_i = (J(2^m,u_i^m))_{m\in\mathbb Z}, \quad 2 \le i \le n \,.
\]
This change of variables defines a bijection on $\mathbb{Z}^{n-1}$ and hence preserves the summation. Since
\[
\|\tau_\nu\|_{E_i \to E_i} = \varphi_{E_i}(2^\nu), \quad\, 2 \leq i \leq n\,,
\]
and
\[
2^\nu \leq \bigg(\frac{\sigma_1}{\sigma_0}\bigg)^{1/(n-1)} < 2^{\nu+1}\,,
\]
the quasi-concavity of $\varphi_{E_i}$ implies that
\[
\varphi_{E_i}(2^\nu)
\leq
\varphi_{E_i}\bigg(\bigg(\frac{\sigma_1}{\sigma_0}\bigg)^{1/(n-1)}\bigg)
\leq
2\,\varphi_{E_i}(2^\nu)\,,
\quad\, 2 \leq i \leq n\,.
\]
Taking the infimum over all decompositions $x_i = \sum u^i_m$, $2 \leq i \leq n$, we obtain
\begin{eqnarray*}
\lefteqn{\|T(x,x_2,\ldots,x_n)\|_{K_E(\yo)}} \nonumber \\
&\leq& C \sigma_0
\varphi_{E_2}\bigg(\bigg(\frac{\sigma_1}{\sigma_0}\bigg)^{1/(n-1)}\bigg)
\cdots
\varphi_{E_n}\bigg(\bigg(\frac{\sigma_1}{\sigma_0}\bigg)^{1/(n-1)}\bigg) \nonumber \\
&& \quad\, \times \|x\|_{K_{E_1}(\xo_1)} \|x_2\|_{J_{E_2}(\xo_2)} \cdots \|x_n\|_{J_{E_n}(\xo_n)}\,.
\end{eqnarray*}
It follows that
\[
\|T\|
\leq
C \sigma_0
\varphi_{E_2}\bigg(\bigg(\frac{\sigma_1}{\sigma_0}\bigg)^{1/(n-1)}\bigg)
\cdots
\varphi_{E_n}\bigg(\bigg(\frac{\sigma_1}{\sigma_0}\bigg)^{1/(n-1)}\bigg)\,.
\]

If $\|T\|_j = 0$ for $j = 0$ or $j = 1$, then, letting $\sigma_j \to 0$ and using the corresponding properties of the functions $\varphi_{E_i}$, we obtain $\|T\| = 0$. In the case $\|T\|_j > 0$ for $j = 0,1$, we take $\sigma_j = (1+\varepsilon)\|T\|_j$ for some $\varepsilon > 0$ to get
\[
\|T\|
\leq
C \|T\|_0
\varphi_{E_2}\bigg(\bigg(\frac{\|T\|_1}{\|T\|_0}\bigg)^{1/(n-1)}\bigg)
\cdots
\varphi_{E_n}\bigg(\bigg(\frac{\|T\|_1}{\|T\|_0}\bigg)^{1/(n-1)}\bigg)\,,
\]
which completes the proof.
\end{proof}

The next two corollaries illustrate Theorem \ref{conv} in two important situations. We first record the bilinear case of Theorem \ref{conv}, which follows 
immediately from the fact that for $n=2$ one has $1/(n-1)=1$.

\begin{corollary}
Under the assumptions of Theorem $\ref{conv}$, in the bilinear case $n=2$ and for $\|T\|_0>0$ and $\|T\|_1>0$, the restriction of $T$ is bounded from 
$K_{E_1}(\xo_1) \times J_{E_2}(\xo_2)$ to $K_E(\yo)$, and
\[
\|T\| \leq C \|T\|_0 \varphi_{E_2}\bigg(\frac{\|T\|_1}{\|T\|_0}\bigg)\,.
\]
If $\|T\|_0=0$ or $\|T\|_1=0$, then $\|T\|=0$.
\end{corollary}

\medskip

We close this section by showing that, in the classical Lions--Peetre setting, Theorem \ref{conv} recovers Zafran's multilinear interpolation theorem \cite{Zafran}, and that the corresponding condition on the interpolation parameters is sharp. For more general multilinear convolution inequalities, we refer to \cite{Oberlin}. To this end, we use the following lemma on weighted convolution. Although it is presumably known, we include a proof for completeness.

\begin{lemma}
Let $n \geq 2$, let $1 \leq p_1,\ldots,p_n,q \leq \infty$, and let $\theta,\theta_1,\ldots,\theta_n \in (0,1)$. Then convolution defines a bounded multilinear operator from $\ell_{p_1}(2^{-\theta_1 m}) \times \cdots \times \ell_{p_n}(2^{-\theta_n m})$ to $\ell_q(2^{-\theta m})$ if and only if $\theta_1=\cdots=\theta_n=\theta$ and
\[
0 \leq \frac1q \leq \frac1{p_1}+\cdots+\frac1{p_n}-(n-1)\,.
\]
\end{lemma}

\begin{proof}
Assume first that $\theta_1=\cdots=\theta_n=\theta$ and
\[
0 \leq \frac1q \leq \frac1{p_1}+\cdots+\frac1{p_n}-(n-1)\,.
\]
For $1 \leq i \leq n$, set $a^i_m := 2^{-\theta m}\xi^i_m$. Then $\xi^i_m = 2^{\theta m} a^i_m$, and for every $k \in \mathbb{Z}$ one has
\[
2^{-\theta k}(\xi^1 * \cdots * \xi^n)_k = (a^1 * \cdots * a^n)_k\,.
\]
Hence
\[
\|\xi^1 * \cdots * \xi^n\|_{\ell_q(2^{-\theta m})} = \|a^1 * \cdots * a^n\|_{\ell_q},
\qquad
\|\xi^i\|_{\ell_{p_i}(2^{-\theta m})} = \|a^i\|_{\ell_{p_i}}, \quad 1 \leq i \leq n.
\]
Thus the desired weighted estimate follows from the discrete multilinear Young inequality, which is obtained 
by iterating the classical bilinear Young inequality on $\mathbb{Z}$.

Conversely, assume that convolution defines a bounded multilinear operator from 
$\ell_{p_1}(2^{-\theta_1 m}) \times \cdots \times \ell_{p_n}(2^{-\theta_n m})$ to $\ell_q(2^{-\theta m})$. Fix $i \in \{1,\ldots,n\}$ 
and $N \in \mathbb{Z}$. Taking $\xi^i=\delta_N$ and $\xi^j=\delta_0$ for $j \neq i$, we obtain $\delta_N * \delta_0 * \cdots * \delta_0 = \delta_N$, 
and therefore
\[
2^{-\theta N} = \|\delta_N\|_{\ell_q(2^{-\theta m})}
\leq
C \|\delta_N\|_{\ell_{p_i}(2^{-\theta_i m})}\prod_{j\neq i}\|\delta_0\|_{\ell_{p_j}(2^{-\theta_j m})}
= C\,2^{-\theta_i N}.
\]
Letting $N \to \infty$, we get $\theta_i \leq \theta$, while letting $N \to -\infty$, we get $\theta \leq \theta_i$. Hence $\theta_i=\theta$. Since $i\in \{1, \ldots, n\}$ 
is arbitrary, it follows that $\theta_1=\cdots=\theta_n=\theta$.

With this information, the weighted estimate reduces to the unweighted one:
\[
\|a^1 * \cdots * a^n\|_{\ell_q} \leq C \|a^1\|_{\ell_{p_1}} \cdots \|a^n\|_{\ell_{p_n}}.
\]
For $N \in \mathbb{N}$, let $u_N := \chi_{\{0,\ldots,N-1\}}$. Then $\|u_N\|_{\ell_{p_i}} = N^{1/p_i}$ for $1 \leq i \leq n$. 
Moreover, for every $k$ with $\lfloor (N-1)/2 \rfloor \leq k \leq N-1$, the restriction $m_j \leq N-1$ is irrelevant, since every nonnegative solution of
\[
m_1+\cdots+m_n=k
\]
automatically satisfies $m_j \leq k \leq N-1$ for each $j$. Thus
\[
(u_N * \cdots * u_N)_k = \#\{(m_1,\ldots,m_n)\in \mathbb{N}_0^n : m_1+\cdots+m_n=k\} = \binom{k+n-1}{n-1}.
\]
Hence there exists a constant $c_n>0$ such that $(u_N * \cdots * u_N)_k \geq c_n N^{n-1}$ for all such $k$. Since there are at least $N/2$ such integers $k$, we get
\[
\|u_N * \cdots * u_N\|_{\ell_q} \geq c'_n N^{n-1+1/q},
\]
with the usual interpretation $1/q=0$ when $q=\infty$. Therefore boundedness yields
\[
N^{n-1+1/q} \leq C N^{1/p_1+\cdots+1/p_n}.
\]

Letting $N \to \infty$, we obtain
\[
\frac1q \leq \frac1{p_1}+\cdots+\frac1{p_n}-(n-1)\,.
\]
This completes the proof.
\end{proof}

As an immediate consequence of Theorem \ref{conv} and the preceding lemma, we recover the classical multilinear Zafran estimate.

\begin{corollary}
Assume, in addition to the hypotheses of Theorem {\rm\ref{conv}}, that
\[
1 \leq p_1,\ldots,p_n,q \leq \infty
\]
and
\[
E_1 := \ell_{p_1}(2^{-m\theta}), \qquad E_i := \ell_{p_i}(2^{-m\theta}), \quad\, 2 \leq i \leq n, \qquad E := \ell_q(2^{-m\theta}),
\]
where $0<\theta<1$ and
\[
0 \leq \frac1q \leq \frac1{p_1} + \cdots + \frac1{p_n} - n + 1\,.
\]
Then the restriction of $T$ is bounded from $(\xo_1)_{\theta,p_1} \times \cdots \times (\xo_n)_{\theta,p_n}$ to $(\yo)_{\theta,q}$, and
\[
\|T\| \leq C \|T\|_0^{1-\theta} \|T\|_1^\theta\,.
\]
\end{corollary}

\begin{proof}
For $1 \leq i \leq n$, let $E_i := \ell_{p_i}(2^{-m\theta})$. Then
\[
K_{E_1}(\xo_1) = (\xo_1)_{\theta,p_1}, \qquad J_{E_i}(\xo_i) = (\xo_i)_{\theta,p_i}, \quad\, 2 \leq i \leq n\,,
\]
and
\[
K_E(\yo) = (\yo)_{\theta,q}\,,
\]
with equivalence of norms. Moreover, by the preceding lemma, condition \eqref{eq1A} is satisfied with $\gamma=1$. Since for $2 \leq i \leq n$ one has
\[
\varphi_{E_i}(t) = t^\theta, \quad\, t>0\,,
\]
it follows from Theorem \ref{conv} that
\begin{align*}
\|T\|
&\leq
C \|T\|_0 \prod_{i=2}^n \varphi_{E_i}\bigg(\bigg(\frac{\|T\|_1}{\|T\|_0}\bigg)^{1/(n-1)}\bigg) \\
&=
C \|T\|_0 \prod_{i=2}^n \bigg(\frac{\|T\|_1}{\|T\|_0}\bigg)^{\theta/(n-1)}
=
C \|T\|_0^{1-\theta} \|T\|_1^\theta.
\end{align*}
This completes the proof.
\end{proof}

The preceding corollary is sharp with respect to the interpolation parameters. More precisely, we have the following characterization.

\begin{corollary}
Let $n \geq 2$, let $\theta,\theta_1,\ldots,\theta_n \in (0,1)$, and let $1 \leq p_1,\ldots,p_n,q \leq \infty$. Then the following assertions are equivalent.
\begin{itemize}
\item[{\rm(i)}] One has
\[
\theta_1=\cdots=\theta_n=\theta
\]
and
\[
0 \leq \frac1q \leq \frac1{p_1}+\cdots+\frac1{p_n}-(n-1)\,.
\]

\item[{\rm(ii)}] For all Banach couples $\xo_i=(X_i^0,X_i^1)$, $1 \leq i \leq n$, and $\yo=(Y^0,Y^1)$, every bounded $n$-linear operator
\[
T \colon \xo_1 \times \cdots \times \xo_n \to \yo
\]
restricts to a bounded operator
\[
T \colon (X_1^0,X_1^1)_{\theta_1,p_1} \times \cdots \times (X_n^0,X_n^1)_{\theta_n,p_n}
\to
(Y^0,Y^1)_{\theta,q}\,.
\]
\end{itemize}
\end{corollary}

\begin{proof}
Assume {\rm(ii)}. For each $1 \leq i \leq n$ and $j \in \{0,1\}$, let
\[
X_i^j := \ell_1(2^{-jm}), \qquad Y^j := \ell_1(2^{-jm})\,.
\]
Then $(X_i^0, X_i^1)$ and $(Y^0, Y^1)$ are Banach couples, and convolution defines a bounded multilinear operator
\[
T \colon \xo_1 \times \cdots \times \xo_n \to \yo\,.
\]
Moreover, by the well-known interpolation formula for the classical real method,
\[
(X_i^0,X_i^1)_{\theta_i,p_i} = \ell_{p_i}(2^{-\theta_i m}), \qquad (Y^0,Y^1)_{\theta,q} = \ell_q(2^{-\theta m})\,,
\]
up to equivalence of norms. Hence {\rm(ii)} implies that convolution defines a bounded multilinear operator
\[
\ell_{p_1}(2^{-\theta_1 m}) \times \cdots \times \ell_{p_n}(2^{-\theta_n m})
\to
\ell_q(2^{-\theta m})\,.
\]
Thus, by the preceding lemma, assertion {\rm(i)} holds.

Conversely, if {\rm(i)} holds, then {\rm(ii)} follows from Corollary 3.2.
\end{proof}

\section{Interpolation of the measure of noncompactness}

The following lemma, which extends the bilinear case treated in \cite[Lemma 3.1]{BC} to the multilinear setting, will be needed below. For the convenience of the reader, we provide a proof.

\begin{lemma}\label{lemma}
Let $X_1,\ldots,X_n$, $Y$, and $Z$ be quasi-Banach spaces, and let $V_i$ be a dense subspace of $X_i$ for each $1 \leq i \leq n$. Given $T \in L(X_1,\ldots,X_n;Y)$, let $\tilde{\beta} = \tilde{\beta}(T \colon X_1 \times \cdots \times X_n \to Y)$, and assume that there exists a sequence of operators $(S_m) \subset L(Y,Z)$ such that $\gamma := \sup_{m\ge 1}\|S_m\|<\infty$ and
\[
\|S_mT(u_1,\ldots,u_n)\|_Z \to 0 \quad\, \text{as\,\, $m\to\infty$}
\]
for all $(u_1,\ldots,u_n)\in V_1\times\cdots\times V_n$. If $\tilde{\beta}>0$, then there exist a constant $C>0$, independent of $T$, and an integer $N\in\mathbb N$ such that
\[
\sup_{m\ge N}\|S_mT\|\le C\tilde{\beta}\,.
\]
Furthermore, if $\tilde{\beta}=0$, then $\|S_mT\|\to 0$ as $m\to\infty$.
\end{lemma}

\begin{proof}
Let $\sigma > \tilde{\beta}$ be fixed. By the definition of $\tilde{\beta}(T)$, there exists a finite set $\{w_1,\ldots,w_s\} \subset Y$ such that
\[
T(B_{X_1} \times \cdots \times B_{X_n}) \subset \bigcup_{k=1}^s (w_k + \sigma B_Y)\,.
\]
Discarding empty sets if necessary, we may assume that
\[
(w_k + \sigma B_Y) \cap T(B_{X_1} \times \cdots \times B_{X_n}) \neq \emptyset
\]
for every $1 \leq k \leq s$. Hence, for each $k \in \{1,\ldots,s\}$, we can find
\[
(x_1^k,\ldots,x_n^k) \in B_{X_1} \times \cdots \times B_{X_n}
\]
such that
\[
T(x_1^k,\ldots,x_n^k) \in w_k + \sigma B_Y\,.
\]
Therefore, for any $(x_1,\ldots,x_n) \in B_{X_1} \times \cdots \times B_{X_n}$, there exists $k \in \{1,\ldots,s\}$ such that
\[
T(x_1,\ldots,x_n) \in w_k + \sigma B_Y\,.
\]
Consequently,
\begin{eqnarray*}
\lefteqn{\|T(x_1,\ldots,x_n) - T(x_1^k,\ldots,x_n^k)\|_Y} \nonumber \\
&\leq& c_Y \big( \|T(x_1,\ldots,x_n) - w_k\|_Y + \|w_k - T(x_1^k,\ldots,x_n^k)\|_Y \big) \\
&\leq& 2 c_Y \sigma\,.
\end{eqnarray*}
Thus
\[
T(B_{X_1} \times \cdots \times B_{X_n})
\subset
\bigcup_{k=1}^s \big(T(x_1^k,\ldots,x_n^k) + 2 c_Y \sigma B_Y\big)\,.
\]

If $T=0$, then the conclusion is obvious. Thus we may assume that $\|T\|>0$.
Since $V_1$ is dense in $X_1$, there exists $u_1^k \in V_1$ such that
\[
\|x_1^k - u_1^k\|_{X_1} \leq \alpha_1 := \frac{\sigma}{n c_Y^{n-1} \|T\|}\,.
\]
By the density of $V_2$ in $X_2$, we can find $u_2^k \in V_2$ such that
\[
\|x_2^k - u_2^k\|_{X_2} \leq \alpha_2 := \frac{1}{c_{X_1}} \frac{\alpha_1}{1+\alpha_1}\,.
\]
Proceeding inductively, we choose $u_i^k \in V_i$, $2 \leq i \leq n$, such that
\[
\|x_i^k - u_i^k\|_{X_i} \leq \alpha_i\,,
\]
where
\[
\alpha_i := \frac{1}{c_{X_{i-1}}} \frac{\alpha_{i-1}}{1+\alpha_{i-1}}\,, \qquad 2 \leq i \leq n\,.
\]
Consequently, for each $1 \leq i \leq n$, we obtain
\[
\|u_i^k\|_{X_i}
\leq
c_{X_i} \big( \|u_i^k - x_i^k\|_{X_i} + \|x_i^k\|_{X_i} \big)
\leq
c_{X_i}(1+\alpha_i)\,.
\]

Now observe that
\[
\|u_1^k\|_{X_1} \|x_2^k - u_2^k\|_{X_2}
\leq
c_{X_1} (1 + \alpha_1) \frac{1}{c_{X_1}} \frac{\alpha_1}{1 + \alpha_1}
=
\alpha_1\,,
\]
and
\[
\|u_1^k\|_{X_1} \|u_2^k\|_{X_2} \|x_3^k - u_3^k\|_{X_3}
\leq
c_{X_1} (1 + \alpha_1) c_{X_2} (1 + \alpha_2)
\frac{1}{c_{X_2}} \frac{\alpha_2}{1 + \alpha_2}
=
\alpha_1\,.
\]
Proceeding inductively, we obtain
\[
\|u_1^k\|_{X_1} \cdots \|u_{i-1}^k\|_{X_{i-1}} \|x_i^k - u_i^k\|_{X_i}
\leq
\alpha_1\,,
\quad\, 2 \leq i \leq n\,.
\]
Since $(x_1^k,\ldots,x_n^k) \in B_{X_1} \times \cdots \times B_{X_n}$, we also have
\[
\|x_i^k\|_{X_i} \leq 1\,,
\quad\, 1 \leq i \leq n\,.
\]
Applying the identity
\begin{align*}
& T(v_1 - u_1, v_2, \ldots, v_n) + T(u_1, v_2 - u_2, v_3, \ldots, v_n) \\
& + T(u_1, u_2, v_3 - u_3, \ldots, v_n) + \cdots + T(u_1, u_2, \ldots, u_{n-1}, v_n - u_n) \\
& = T(v_1, v_2, \ldots, v_n) - T(u_1, u_2, \ldots, u_n)\,,
\end{align*}
we obtain
\begin{eqnarray*}
\lefteqn{\|T(x_1^k,\ldots,x_n^k) - T(u_1^k,\ldots,u_n^k)\|_Y} \nonumber \\
&\leq& c_Y^{n-1} \|T\| \bigg(
\|x_1^k - u_1^k\|_{X_1} \|x_2^k\|_{X_2} \cdots \|x_n^k\|_{X_n} \nonumber \\
&& \qquad\qquad\qquad\quad\, + \|u_1^k\|_{X_1} \|x_2^k - u_2^k\|_{X_2} \|x_3^k\|_{X_3} \cdots \|x_n^k\|_{X_n} \nonumber \\
&& \qquad\qquad\qquad\quad\, + \cdots + \|u_1^k\|_{X_1} \cdots \|u_{n-1}^k\|_{X_{n-1}} \|x_n^k - u_n^k\|_{X_n}
\bigg) \nonumber \\
&\leq& n c_Y^{n-1} \|T\| \alpha_1 = \sigma\,.
\end{eqnarray*}
Consequently,
\begin{eqnarray*}
\lefteqn{\|T(x_1,\ldots,x_n) - T(u_1^k,\ldots,u_n^k)\|_Y} \nonumber \\
&\leq& c_Y \big( \|T(x_1,\ldots,x_n) - T(x_1^k,\ldots,x_n^k)\|_Y \nonumber \\
&& \qquad\qquad\quad\, + \|T(x_1^k,\ldots,x_n^k) - T(u_1^k,\ldots,u_n^k)\|_Y \big) \nonumber \\
&\leq& c_Y(2 c_Y + 1)\sigma\,,
\end{eqnarray*}
and hence
\[
T(B_{X_1} \times \cdots \times B_{X_n})
\subset
\bigcup_{k=1}^s \big(T(u_1^k,\ldots,u_n^k) + c_Y(2 c_Y + 1)\sigma B_Y\big)\,.
\]

We now define
\[
C := 2 c_Z \big(\gamma c_Y(2 c_Y + 1) + 1\big)\,,
\]
and choose $N \in \mathbb N$ such that
\[
\|S_m T(u_1^k,\ldots,u_n^k)\|_Z \leq \sigma,
\quad\, 1 \leq k \leq s,\quad\, m \geq N\,.
\]
Then, by the above inclusion, for any $(x_1,\ldots,x_n) \in B_{X_1} \times \cdots \times B_{X_n}$ there exists $k \in \{1,\ldots,s\}$ such that
\[
\|T(x_1,\ldots,x_n) - T(u_1^k,\ldots,u_n^k)\|_Y \leq c_Y(2 c_Y + 1)\sigma\,.
\]
Thus, for each $m \geq N$, using the bound $\|S_m\| \le \gamma$, we have
\begin{eqnarray*}
\lefteqn{\|S_m T(x_1,\ldots,x_n)\|_Z} \nonumber \\
&\leq& c_Z \big( \|S_m T(x_1,\ldots,x_n) - S_m T(u_1^k,\ldots,u_n^k)\|_Z + \|S_m T(u_1^k,\ldots,u_n^k)\|_Z \big) \nonumber \\
&\leq& c_Z \big( \|S_m\| \|T(x_1,\ldots,x_n) - T(u_1^k,\ldots,u_n^k)\|_Y + \sigma \big) \nonumber \\
&\leq& c_Z \big( \gamma c_Y(2 c_Y + 1)\sigma + \sigma \big)
\leq
\frac{C\sigma}{2}\,.
\end{eqnarray*}
Hence
\[
\sup_{m \geq N} \|S_m T\| \leq \frac{C\sigma}{2}\,.
\]

If $\tilde{\beta} > 0$, we now choose $\sigma = 2\tilde{\beta}$ and obtain
\[
\sup_{m \geq N} \|S_m T\| \leq C \tilde{\beta}\,.
\]
If $\tilde{\beta} = 0$, then the same argument applies for every $\sigma > 0$, and therefore
\[
\|S_m T\| \to 0 \quad\, \text{as\,\, $m \to \infty$}\,.
\]
This completes the proof.
\end{proof}

\medskip

Given a quasi-Banach couple $\yo = (Y^0,Y^1)$, for each $m\in \mathbb{Z}$ we define $W_m := (Y^0 + Y^1, K(2^m,\cdot))$. Let $Q$ be the linear operator that assigns to any $y \in Y^0 + Y^1$ the sequence $Qy = (y)_m$, where all coordinates are equal to $y$. For $j = 0,1$, it follows that the mapping
\[
Q\colon Y^j \to \ell_\infty(2^{-mj}W_m)
\]
is bounded with norm at most $1$.

\begin{lemma}\label{lemma2}
Let $\xo_i = (X_i^0,X_i^1)$ for $1 \leq i \leq n$, and let $\yo = (Y^0,Y^1)$ be quasi-Banach couples. Let $T \in L(\xo_1,\ldots,\xo_n;\yo)$, and define
\[
\tilde{\beta}_j = \tilde{\beta}(T \colon X_1^j \times \cdots \times X_n^j \to Y^j), \quad\, j=0,1\,.
\]
Assume there exist quasi-Banach spaces $Z_1, \ldots, Z_n$ and sequences $(R^i_\nu)_{\nu \geq 1} \subset L(Z_i, X_i^0 \cap X_i^1)$ such that
\[
\sup_{\nu \geq 1} \|R^i_\nu\|_{L(Z_i, X_i^j)} \leq 1, \quad\, j = 0,1, \quad\, 1 \leq i \leq n,
\]
and
\[
\lim_{\nu \to \infty} \|T(R_\nu^1, \ldots, R_\nu^n)\|_{L(Z_1, \ldots, Z_n; Y^0 + Y^1)} = 0\,.
\]
Then, for each $j \in \{0,1\}$, the following assertions hold:
\begin{itemize}
\item[{\rm(a)}] If $\tilde{\beta}_j > 0$, then there exist a constant $C > 0$, independent of $T$, and a subsequence $(\nu_k)$ of $(\nu)_{\nu \geq 1}$ such that
\[
\limsup_{k \to \infty} \|Q T(R_{\nu_k}^{1}, \ldots, R_{\nu_k}^{n})\|_{L(Z_1, \ldots, Z_n; \ell_\infty(2^{-mj} W_m))} \leq C \tilde{\beta}_j.
\]

\item[{\rm(b)}] If $\tilde{\beta}_j = 0$, then there exists a subsequence $(\nu_k)$ such that
\[
\lim_{k \to \infty} \|Q T(R_{\nu_k}^{1}, \ldots, R_{\nu_k}^{n})\|_{L(Z_1, \ldots, Z_n; \ell_\infty(2^{-mj} W_m))} = 0.
\]
\end{itemize}
\end{lemma}

\begin{proof}
Fix $j \in \{0,1\}$. We consider the following diagram:

\medskip

\begin{tikzcd}
Z_1 \times \cdots \times Z_n \arrow[r,"(R_\nu^1 {,} \ldots {,} R_\nu^n) "] &[2em] X_1^j \times \cdots \times X_n^j \ar[r,"T"] &[0.1em] Y^j \ar[r,"Q"] &[0.1em] \ell^\infty(2^{-mj}W_m)
\end{tikzcd}

\medskip

\noindent
Since for each $\nu$,
\[
\|Q T(R_\nu^{1}, \ldots, R_\nu^n)\| \leq \|Q\| \|T\|_j \|R_\nu^1\|_{L(Z_1,X_1^j)} \cdots \|R_\nu^n\|_{L(Z_n,X_n^j)}\,,
\]
we have
\[
\sup_{\nu \in \mathbb{N}} \|Q T(R_\nu^{1}, \ldots, R_\nu^n)\| \leq \|T\|_j\,.
\]
Passing to a subsequence, we may assume that
\[
\lim_{\nu \to \infty} \|Q T(R_{\nu}^{1}, \ldots, R_{\nu}^n)\|_{L(Z_1,\ldots,Z_n;\ell_\infty(2^{-mj}W_m))} = \lambda \geq 0\,.
\]
For each $\nu$, choose $(z_{\nu}^1,\ldots,z_{\nu}^n) \in B_{Z_1}\times\cdots\times B_{Z_n}$ such that
\[
\|Q T(R_{\nu}^1 z_{\nu}^1, \ldots, R_{\nu}^n z_{\nu}^n)\|_{\ell_\infty(2^{-mj}W_m)}
\geq
\|Q T(R_{\nu}^{1}, \ldots, R_{\nu}^n)\|_{L(Z_1,\ldots,Z_n;\ell_\infty(2^{-mj}W_m))} - \frac{1}{\nu}\,.
\]
Since $(z_{\nu}^1,\ldots,z_{\nu}^n)\in B_{Z_1}\times\cdots\times B_{Z_n}$, we also have
\[
\|Q T(R_{\nu}^1 z_{\nu}^1, \ldots, R_{\nu}^n z_{\nu}^n)\|_{\ell_\infty(2^{-mj}W_m)}
\leq
\|Q T(R_{\nu}^{1}, \ldots, R_{\nu}^n)\|_{L(Z_1,\ldots,Z_n;\ell_\infty(2^{-mj}W_m))}.
\]
Hence
\[
\lim_{\nu \to \infty} \|Q T(R_{\nu}^1 z_{\nu}^1, \ldots, R_{\nu}^n z_{\nu}^n)\|_{\ell_\infty(2^{-mj}W_m)} = \lambda\,.
\]

Let $\sigma > \tilde{\beta}_j$. Then there exists a $\sigma$-net $\{y_1^j, \ldots, y_s^j\} \subset Y^j$ such that
\[
T(B_{X_1^j} \times \cdots \times B_{X_n^j}) \subset \bigcup_{k=1}^s (y_k^j + \sigma B_{Y^j})\,.
\]
Since $\|R_\nu^i\|_{L(Z_i,X_i^j)} \leq 1$ and $z_\nu^i \in B_{Z_i}$, we have $R_\nu^i z_\nu^i \in B_{X_i^j}$ for each $i$. Thus, again passing to a subsequence, we may assume that for some $k$, $1 \leq k \leq s$, one has
\begin{equation}
\label{eq5A}
T(R_{\nu}^1 z_{\nu}^1, \ldots, R_{\nu}^n z_{\nu}^n) \in y_k^j + \sigma B_{Y^j}
\end{equation}
for all $\nu$.

Let $m \in \mathbb{Z}$. Since
\[
\lim_{\nu \to \infty} \|T(R_{\nu}^1, \ldots, R_{\nu}^n)\|_{L(Z_1, \ldots, Z_n; Y^0 + Y^1)} = 0\,,
\]
there exists $l$ such that
\[
\|T(R_{l}^{1}, \ldots, R_{l}^n)\|_{L(Z_1, \ldots, Z_n; Y^0 + Y^1)} \leq \sigma 2^{mj} \max(1, 2^m)^{-1}\,.
\]
Then
\begin{align*}
2^{-mj} K(2^m, y_k^j)
&\leq c_{\yo} \bigg(
2^{-mj} K\big(2^m, y_k^j - T(R_{l}^1 z_{l}^1, \ldots, R_{l}^n z_{l}^n)\big) \\
&\quad\, + 2^{-mj} K\big(2^m, T(R_{l}^1 z_{l}^1, \ldots, R_{l}^n z_{l}^n)\big)
\bigg) \\
&\leq c_{\yo} \bigg(
\|y_k^j - T(R_{l}^1 z_{l}^1, \ldots, R_{l}^n z_{l}^n)\|_{Y^j} \\
&\quad\, + 2^{-mj} \max(1, 2^m)
\|T(R_{l}^1 z_{l}^1, \ldots, R_{l}^n z_{l}^n)\|_{Y^0 + Y^1}
\bigg) \\
&\leq 2 c_{\yo} \sigma\,.
\end{align*}
Hence
\[
\|Q y_k^j\|_{\ell^\infty(2^{-mj} W_m)} \leq 2 c_{\yo} \sigma\,.
\]

Since $\|Q\| \leq 1$, using \eqref{eq5A}, we obtain for every $\nu$,
\begin{eqnarray*}
\|Q T(R_{\nu}^1 z_{\nu}^{1}, \ldots, R_{\nu}^n z_{\nu}^n)\|_{\ell^\infty(2^{-mj} W_m)}
&\leq&
c_{\yo} \big( \|Q T(R_{\nu}^1 z_{\nu}^1, \ldots, R_{\nu}^n z_{\nu}^n) - Q y_k^j\|_{\ell^\infty(2^{-mj} W_m)} \\
&& \qquad\qquad + \|Q y_k^j\|_{\ell^\infty(2^{-mj} W_m)} \big) \\
&\leq&
c_{\yo} \big( \|Q\| \|T(R_{\nu}^1 z_{\nu}^1, \ldots, R_{\nu}^n z_{\nu}^n) - y_k^j\|_{Y^j} + 2 c_{\yo} \sigma \big) \\
&\leq&
c_{\yo} \big( \sigma + 2 c_{\yo} \sigma \big)\,.
\end{eqnarray*}
Thus, if $C = 2 c_{\yo}(1 + 2 c_{\yo})$, then
\[
\|Q T(R_{\nu}^1 z_{\nu}^{1}, \ldots, R_{\nu}^n z_{\nu}^n)\|_{\ell^\infty(2^{-mj} W_m)} \leq C \sigma / 2
\]
for every $\nu$. Passing to the limit as $\nu \to \infty$, we obtain
\[
\lambda \leq C \sigma / 2\,.
\]

If $\tilde{\beta}_j > 0$, taking $\sigma = 2 \tilde{\beta}_j$, we obtain
\[
\lim_{\nu \to \infty} \|Q T(R_{\nu}^{1}, \ldots, R_{\nu}^n)\|_{L(Z_1, \ldots, Z_n; \ell^\infty(2^{-mj} W_m))} \leq C \tilde{\beta}_j\,.
\]
If $\tilde{\beta}_j = 0$, then $\sigma > 0$ is arbitrary, and therefore $\lambda = 0$. Hence
\[
\lim_{\nu \to \infty} \|Q T(R_{\nu}^{1}, \ldots, R_{\nu}^n)\|_{L(Z_1, \ldots, Z_n; \ell^\infty(2^{-mj} W_m))} = 0\,.
\]
This completes the proof.
\end{proof}

\medskip

In what follows, for each integer $N \geq 0$, we define $\mathbb{R}_N := \prod_{k=-N}^N \mathbb{R}$. Thus, for a given $x \in \mathbb{R}_N$, we have
\[
x = (x_{-N}, \ldots, x_N) = \sum_{k=-N}^{N} x_k e_k\,,
\]
where $e_k = (0, \ldots, 0, 1, 0, \ldots, 0)$, with $1$ in the $k$-th position.

Let $E \subset \omega(\mathbb{Z})$ be a quasi-Banach $p$-normed sequence lattice, where $0 < p \leq 1$. For each integer $N \geq 0$, define on $\mathbb{R}_N$ the $p$-norm
\[
\|x\|_p = \Big( \sum_{k=-N}^{N} |x_k|^p \Big)^{1/p}, \quad x \in \mathbb{R}_N\,.
\]
Let $S \colon \mathbb{R}_N \to E$ be the operator given by $S(x) := \widetilde{x} = (\widetilde{x}_k)$, where
\[
\widetilde{x}_k =
\begin{cases}
x_k, & \text{if } |k| \leq N\,, \\
0, & \text{if } |k| > N\,.
\end{cases}
\]
We also define on $\mathbb{R}_N$ the $p$-norm $\|x\|_{\widetilde{E}} := \|\widetilde{x}\|_E$. Since $\mathbb{R}_N$ is finite-dimensional, the $p$-norms $\|x\|_{\widetilde{E}}$ and $\|x\|_p$ are equivalent on $\mathbb{R}_N$. Consequently, the unit ball $B_{\widetilde{E},N}$ of $(\mathbb{R}_N,\|\cdot\|_{\widetilde{E}})$ is compact. Hence, for every $\varepsilon > 0$, there exists a finite set $\{x^1,\ldots,x^s\} \subset \mathbb{R}_N$ such that for every $x \in B_{\widetilde{E},N}$ one has
\begin{equation}\label{net}
\min_{1 \leq k \leq s} \|x - x^k\|_{\widetilde{E}} \leq \varepsilon\,.
\end{equation}

The following result from \cite{CLM1} will be useful in the subsequent discussion.

\begin{lemma}\label{lemma1}
Let $0 < q_0, q_1 \leq \infty$, and let $(W_m)_{m \in \mathbb{Z}}$ be a sequence of quasi-Banach spaces such that the quasi-triangle constant is the same for all $W_m$
 Then the pair $\big(\ell_{q_0}(W_m), \ell_{q_1}(2^{-m}W_m)\big)$ is a $p$-normed quasi-Banach couple for some $0 < p \leq 1$. Moreover, 
if $E$ is a quasi-Banach lattice satisfying condition \eqref{cite1}, then
\[
(\ell_{q_0}(W_m), \ell_{q_1}(2^{-m}W_m))_E = E(W_m)
\]
with equivalent quasi-norms.
\end{lemma}

\vspace{2 mm}

\begin{theorem}\label{main}
Let $n \geq 2$. Let $\xo_i = (X_i^0, X_i^1)$ for $1 \leq i \leq n$ be $p_i$-normed quasi-Banach couples, and let $\yo = (Y^0, Y^1)$ be an $r$-normed quasi-Banach couple, where $0 < p_i, r \leq 1$. Assume that $E$ is a $K$-non-trivial quasi-Banach sequence lattice satisfying \eqref{cite1}, and that, for each $1 \leq i \leq n$, $E_i$ is a $K$-non-trivial, $(p_i,J)$-non-trivial quasi-Banach sequence lattice such that $\Omega_{p_i}$ is bounded on $E_i$ and $E_i$ satisfies the analogue of \eqref{cite1}. Furthermore, suppose that there exists a constant $\gamma > 0$ such that, for all sequences $\xi_1 \in E_1, \ldots, \xi_n \in E_n$, the following inequality holds{\rm:}
\[
\big\|\big(|\xi_1|^r \ast \cdots \ast |\xi_n|^r\big)^{1/r}\big\|_E \leq \gamma \|\xi_1\|_{E_1} \cdots \|\xi_n\|_{E_n}\,.
\]
Let $T \colon \xo_1 \times \cdots \times \xo_n \to \yo$ be an $n$-linear mapping, and for $j = 0,1$ set
\[
\tilde{\beta}_j := \tilde{\beta}(T \colon X_1^j \times \cdots \times X_n^j \to Y^j)\,.
\]
Then the restriction of $T$ is bounded from $K_{E_1}(\xo_1) \times J_{E_2}(\xo_2) \times \cdots \times J_{E_n}(\xo_n)$ to $K_E(\yo)$, and
\[
\widetilde{\beta}(T) = \widetilde{\beta}\big(T \colon K_{E_1}(\xo_1) \times J_{E_2}(\xo_2) \times \cdots \times J_{E_n}(\xo_n) \to K_E(\yo)\big)
\]
satisfies the estimate
\[
\widetilde{\beta}(T) \leq C \tilde{\beta}_0 \varphi_{E_2}\bigg(\bigg(\frac{\tilde{\beta}_1}{\tilde{\beta}_0}\bigg)^{1/(n-1)}\bigg) \cdots \varphi_{E_n}\bigg(\bigg(\frac{\tilde{\beta}_1}{\tilde{\beta}_0}\bigg)^{1/(n-1)}\bigg)\,,
\]
provided that $\widetilde{\beta}_0 > 0$ and $\widetilde{\beta}_1 > 0$, where $C$ is a constant independent of $T${\rm;} otherwise, $\widetilde{\beta}(T)=0$.
\end{theorem}

\medskip

\begin{proof}
Let $p := \min\{p_1,\ldots,p_n,r\}$. Then all spaces $X_i^0$, $X_i^1$, $Y^0$, and $Y^1$ are $p$-normed. Hence the spaces
\[
F_m^i := (X_i^0 \cap X_i^1, J(2^m, \cdot, X_i^0, X_i^1)), \quad\, 1\leq i \leq n-1,
\]
and
\[
G_m := (X_n^0 \cap X_n^1, J(2^m, \cdot, X_n^0, X_n^1)), \quad\, m \in \mathbb{Z},
\]
are also $p$-normed.

We consider the $p$-normed quasi-Banach couples
\[
\vec{F}_i := (\ell_p(F_m^i), \ell_p(2^{-m}F_m^i)), \quad\, 1 \leq i \leq n-1,
\]
and
\[
\vec{G} := (\ell_p(G_m), \ell_p(2^{-m}G_m))\,.
\]
From Lemma \ref{lemma1}, we have identities with equivalent quasi-norms
\[
(\ell_p(F_m^i), \ell_p(2^{-m}F_m^i))_{E_i} = E_i(F_m^i), \quad\, 1 \leq i \leq n-1\,,
\]
and
\[
(\ell_p(G_m), \ell_p(2^{-m}G_m))_{E_n} = E_n(G_m)\,.
\]

For each $1 \leq i \leq n-1$, define a linear mapping $\pi_i \colon E_i(F_m^i) \to J_{E_i}(\xo_i)$ by
\[
\pi_i((u_m)) := \sum_{m=-\infty}^{\infty} u_m, \quad\, (u_m) \in E_i(F_m^i)\,.
\]
Then $\pi_i$ is a bounded surjection with norm at most $1$. Indeed, if $(u_m)\in E_i(F_m^i)$, then $(J(2^m,u_m))\in E_i$, and therefore
\[
\|\pi_i((u_m))\|_{J_{E_i}(\xo_i)} \leq \|(J(2^m,u_m))\|_{E_i} = \|(u_m)\|_{E_i(F_m^i)}\,.
\]
Conversely, by the definition of the quasi-norm on $J_{E_i}(\xo_i)$, if $x\in J_{E_i}(\xo_i)$ and $\|x\|_{J_{E_i}(\xo_i)}<1$, then there exists $(u_m)\in E_i(F_m^i)$ such that
\[
x=\sum_{m=-\infty}^\infty u_m, \quad\, \|(u_m)\|_{E_i(F_m^i)}<1\,.
\]

Similarly, define a linear mapping $\pi_n \colon E_n(G_m) \to J_{E_n}(\xo_n)$ by
\[
\pi_n((u_m)) := \sum_{m=-\infty}^{\infty} u_m, \quad\, (u_m) \in E_n(G_m)\,.
\]
Then, by the same argument, $\pi_n$ is a bounded surjection with norm at most $1$.

For $j = 0,1$, the mappings
\[
\pi_i \colon \ell_p(2^{-mj}F_m^i) \to X_i^j, \quad\, 1 \leq i \leq n-1,
\]
and
\[
\pi_n \colon \ell_p(2^{-mj}G_m) \to X_n^j
\]
are bounded with norm at most $1$.

Now, for each $m\in \mathbb{Z}$, let
\[
W_m := (Y^0 + Y^1, K(2^m, \cdot, Y^0, Y^1))
\]
and
\[
W_\infty := (\ell_\infty(W_m), \ell_\infty(2^{-m}W_m))\,.
\]
Define the linear mapping $Q$ by
\[
Qy := (y)_m = (\ldots, y, y, y, \ldots), \quad\, y \in Y^{0} + Y^{1}\,.
\]
Then $Q\colon K_E(\yo) \to E(W_m)$ is an isometric embedding. We also have that $Q\colon Y^j \to \ell_\infty(2^{-mj}W_m)$ is 
bounded with norm at most $1$, for $j = 0,1$. By Lemma \ref{lemma1}, we obtain
\[
(\ell_\infty(W_m), \ell_\infty(2^{-m}W_m))_E = E(W_m)
\]
with equivalent quasi-norms.

For $j=0,1$, we consider the diagrams
\[
\Big(\prod_{i=1}^{n-1} \ell_p(2^{-mj} F_m^i)\Big) \times \ell_p(2^{-mj} G_m)
\xrightarrow{(\pi_1,\ldots,\pi_n)}
\Big(\prod_{i=1}^{n} X_i^j\Big)
\stackrel{T}{\longrightarrow}
Y^j
\stackrel{Q}{\longrightarrow}
\ell_\infty(2^{-mj} W_m)\,.
\]
Hence, by Theorem \ref{conv}, we obtain
\[
\Big(\prod_{i=1}^{n-1} E_i(F_m^i)\Big) \times E_n(G_m)
\xrightarrow{(\pi_1,\ldots,\pi_n)}
K_{E_1}(\xo_1) \times \Big(\prod_{i=2}^{n} J_{E_i}(\xo_i)\Big)
\stackrel{T}{\longrightarrow}
K_E(\yo)
\stackrel{Q}{\longrightarrow}
E(W_m)\,.
\]
Here we use the fact that $K_{E_1}(\xo_1)=J_{E_1}(\xo_1)$ up to equivalence of quasi-norms.

Observe that the operator $\widehat{T}$ given by
\[
\widehat{T} := Q T (\pi_1, \ldots, \pi_n)
\]
maps $\vec{F}_1 \times \cdots \times \vec{F}_{n-1} \times \vec{G}$ to $W_\infty$ and is bounded. Then, from properties {\rm(v)} and {\rm(vi)} of MNC, and the properties of $Q$ and $\pi_i$, we obtain
\begin{align}
\widetilde{\beta}(T)
&\leq 2 C_E \widetilde{\beta}\big(Q T \colon K_{E_1}(\xo_1) \times J_{E_2}(\xo_2) \times \cdots \times J_{E_n}(\xo_n) \to E(W_m)\big) \nonumber \\
&\leq 2 C_E \widetilde{\beta}\big(\widehat{T} \colon E_1(F_m^1) \times \cdots \times E_{n-1}(F_m^{n-1}) \times E_n(G_m) \to E(W_m)\big)\,.
\end{align}

For all $\nu \in \mathbb{N}$, and for each $i \in \{1, \ldots, n-1\}$, we consider on $\ell_p(2^{-mj} F_m^i)$, $j=0,1$, the bounded operators $R_{i,\nu}$, $R_{i,\nu}^{+}$, and $R_{i,\nu}^{-}$, given for all $(u_m) \in \ell_p(2^{-mj} F_m^i)$ by
\[
R_{i,\nu}(u_m) = (\ldots, 0, 0, u_{-\nu}, \ldots, u_\nu, 0, 0, \ldots)\,,
\]
\[
R_{i,\nu}^{+}(u_m) = (\ldots, 0, 0, \ldots, 0, u_{\nu+1}, u_{\nu+2}, \ldots)\,,
\]
\[
R_{i,\nu}^{-}(u_m) = (\ldots, u_{-\nu-3}, u_{-\nu-2}, u_{-\nu-1}, 0, 0, \ldots)\,.
\]
If $I$ is the identity on $\ell_p(F_m^i) + \ell_p(2^{-m}F_m^i)$, then
\[
I = R_{i,\nu} + R_{i,\nu}^{+} + R_{i,\nu}^{-}, \quad\, \nu \in \mathbb{N}\,.
\]
Also,
\[
R_{i,\nu} \colon \ell_p(2^{-mj} F_m^i) \to \ell_p(2^{-mj} F_m^i), \quad\, j = 0,1,
\]
is bounded, with norm at most $1$. The same holds for $R_{i,\nu}^{+}$ and $R_{i,\nu}^{-}$ for $1 \leq i \leq n-1$. Clearly,
\[
R_{i,\nu}, R_{i,\nu}^{+}, R_{i,\nu}^{-} \colon E_i(F_m^i) \longrightarrow E_i(F_m^i), \quad\, 1 \leq i \leq n-1\,.
\]
Moreover, we have the following boundedness results:
\[
R_{i,\nu} \colon  \ell_p(F^i_m) + \ell_p(2^{-m} F^i_m) \longrightarrow \ell_p(F^i_m) \cap \ell_p(2^{-m} F^i_m)\,,
\]
\[
R_{i,\nu}^{+} \colon \ell_p(F^i_m) \longrightarrow \ell_p(2^{-m} F^i_m)\,,
\]
\[
R_{i,\nu}^{-} \colon \ell_p(2^{-m} F^i_m) \longrightarrow \ell_p(F^i_m)\,.
\]
Additionally, we have the following inequalities for the norms of these operators:
\begin{equation} \label{eq1}
\|R_{i,\nu}\|_{\ell_p(F^i_m) + \ell_p(2^{-m} F^i_m) \to \ell_p(F^i_m) \cap \ell_p(2^{-m} F^i_m)} \leq C_{\mathcal{X}_i} 2^{1/p} 2^{\nu}\,,
\end{equation}

\begin{equation} \label{eq2}
\|R_{i,\nu}^+\|_{\ell_p(F^i_m) \to \ell_p(2^{-m} F^i_m)} = \|R_{i,\nu}^{-}\|_{\ell_p(2^{-m} F^i_m) \to \ell_p(F^i_m)} = 2^{-(\nu + 1)}\,.
\end{equation}
Let $S_\nu, S_\nu^+, S_\nu^-$ be operators defined on $\ell_p(2^{-mj} G_m)$, and $P_\nu, P_\nu^+, P_\nu^-$ be operators defined on $\ell_\infty(2^{-mj} W_m)$, 
similar to the operators $R_\nu, R_\nu^+$, and $R_\nu^-$, respectively.  Note that these operators also satisfy (\ref{eq1}) and (\ref{eq2}).  Now, we decompose 
$\widehat{T}$ as follows:
\small
\begin{align} \label{dec1}
\lefteqn{\widehat{T} = I \widehat{T} = (P_{(n+1)\nu} + P_{(n+1)\nu}^+ + P_{(n+1)\nu}^-) \widehat{T} (I,\ldots,I,I)} \\ \nonumber
&= P_{(n+1)\nu} \widehat{T} (R_{1,4\nu} + R_{1,4\nu}^+ + R_{1,4\nu}^{-}, \ldots, R_{n-1,4\nu} + R_{n-1,4\nu}^+ + R_{n-1,4\nu}^-, S_{4\nu} + S_{4\nu}^+ + S_{4\nu}^-) \\ \nonumber
&\quad + P_{(n+1)\nu}^+ \widehat{T} + P_{(n+1)\nu}^- \widehat{T}\,.
\end{align}
\normalsize

Now we need to analyze each one of the three terms in the last line of (\ref{dec1}). We begin with 
\[
P_{(n+1)\nu} \widehat{T} (R_{1,4\nu} + R_{1,4\nu}^{+} + R_{1,4\nu}^{-}, \ldots, R_{n-1,4\nu} + R_{n-1,4\nu}^{+} + R_{n-1,4\nu}^{-}, S_{4\nu} + S_{4\nu}^{+} + S_{4\nu}^{-})\,,
\]
and since $T$ is $n$-linear, when we expand the first $n-1$ terms, we have:

\begin{itemize}
\item[{\bf (i)}] One term with no sign $+$ or $-$ in the $R$ operators, where
\begin{align} \label{dec2}
& P_{(n+1)\nu} \hat{T} (R_{1,4\nu}, \ldots, R_{n-1,4\nu}, S_{4\nu} + S_{4\nu}^+ + S_{4\nu}^-) 
= P_{(n+1)\nu} \hat{T} (R_{1,4\nu}, \ldots, R_{n-1,4\nu}, S_{4\nu}) \\ \nonumber
& \quad + P_{(n+1)\nu} \hat{T} (R_{1,4\nu}, \ldots, R_{n-1,4\nu}, S_{4\nu}^{+}) 
+ P_{(n+1)\nu} \hat{T} (R_{1,4\nu}, \ldots, R_{n-1,4\nu}, S_{4\nu}^{-})\,.
\end{align}
\item[{\bf (ii)}] One term where we have only $R^+$ operators in the first $n-1$ positions, that is,
\begin{align} \label{dec3}
& P_{(n+1)\nu} \widehat{T} (R_{1,4\nu}^{+}, \ldots, R_{n-1,4\nu}^{+}, S_{4\nu} + S_{4\nu}^{+} + S_{4\nu}^{-}) \\ \nonumber
& \quad = P_{(n+1)\nu} \widehat{T} (R_{1,4\nu}^{+}, \ldots, R_{n-1,4\nu}^+, S_{4\nu} + S_{4\nu}^+) 
+ P_{(n+1)\nu} \hat{T} (R_{1,4\nu}^+, \ldots, R_{n-1,4\nu}^+, S_{4\nu}^{-}).
\end{align}
\item[{\bf (iii)}] One term where we have only $R^-$ operators in the first $n-1$ positions, that is,
\begin{align} \label{dec4}
& P_{(n+1)\nu} \hat{T} (R_{1,4\nu}^{-}, \ldots, R_{n-1,4\nu}^-, S_{4\nu} + S_{4\nu}^{+} + S_{4\nu}^-) \\ \nonumber 
&\quad = P_{(n+1)\nu} \hat{T} (R_{1,4\nu}^{-}, \ldots, R_{n-1,4\nu}^-, S_{4\nu} + S_{4\nu}^-) 
+ P_{(n+1)\nu} \hat{T} (R_{1,4\nu}^-, \ldots, R_{n-1,4\nu}^{-}, S_{4\nu}^{+})\,.
\end{align}
\item[{\bf (iv)}] Terms with at least one $R$ together with one or more $R^+$ and/or $R^-$. The number of these terms is given by  
\[
\binom{n-1}{1} 2^{n-2} + \binom{n-1}{2} 2^{n-3} + \cdots + \binom{n-1}{n-2} 2 = 3^{n-1} - 2^{n-1} - 1.
\]
\item[{\bf (v)}] Terms having only operators $R^+$ and $R^-$, with at least one of each. The total number of these terms is $2^{n-1} - 2$.
\end{itemize} 

{\bf Step 2}. We begin with {\rm(i)}, considering
\[
P_{(n+1)\nu} \widehat{T} (R_{1,4\nu}, \ldots, R_{n-1,4\nu}, S_{4\nu})\,.
\]
From property {\rm(v)} of MNC, we have
\begin{align*}
& \widetilde{\beta}(P_{(n+1)\nu} \widehat{T} (R_{1,4\nu}, \ldots, R_{n-1,4\nu}, S_{4\nu})) \\
&\leq \|P_{(n+1)\nu}\|_{E(W_m)\to E(W_m)} \|Q\|_{K_E(\yo)\to E(W_m)} \\
&\quad\, \times \widetilde{\beta}\big(T (\pi_1 R_{1,4\nu}, \ldots, \pi_{n-1} R_{n-1,4\nu}, \pi_n S_{4\nu}) \colon E_1(F_m^1) \times \cdots \times E_{n-1}(F_m^{n-1}) \times E_n(G_m) \to K_E(\yo)\big) \\
&\leq C \, \widetilde{\beta}\big(T (\pi_1 R_{1,4\nu}, \ldots, \pi_{n-1} R_{n-1,4\nu}, \pi_n S_{4\nu}) \colon E_1(F_m^1) \times \cdots \times E_{n-1}(F_m^{n-1}) \times E_n(G_m) \to K_E(\yo)\big)\,.
\end{align*}

For each $\nu \in \mathbb{N}$ and each $1 \leq i \leq n$, we apply the above construction with $N=4\nu$ and $E=E_i$. Thus, on $\mathbb{R}_{4\nu}$ we consider the quasi-norm
\[
\|x\|_{\widetilde{E}_i} := \|(\ldots, 0, 0, x_{-4\nu}, \ldots, x_{4\nu}, 0, 0, \ldots)\|_{E_i}, \quad\, x=(x_k)_{k=-4\nu}^{4\nu} \in \mathbb{R}_{4\nu}\,.
\]
Let
\[
\eta := \bigg(\max_{1\leq i \leq n} \bigg\|\sum_{|k| \leq 4\nu} \frac{e_k}{\|e_k\|_{E_i}}\bigg\|_{E_i}\bigg)^{-1}.
\]
By \eqref{net}, there exists an $\eta$-net
\[
\Lambda_1 := \{\lambda^1, \ldots, \lambda^s\} \subset B_{(\mathbb{R}_{4\nu}, \|\cdot\|_{\widetilde{E}_1})}
\]
such that, given $x \in B_{(\mathbb{R}_{4\nu}, \|\cdot\|_{\widetilde{E}_1})}$, there exists $\lambda^d \in \Lambda_1$ with
\[
\|x - \lambda^d\|_{\widetilde{E}_1} \leq \eta\,.
\]
Similarly, for each $2 \leq i \leq n$, let
\[
\Lambda_i := \{\mu_i^1, \ldots, \mu_i^{r_i}\} \subset B_{(\mathbb{R}_{4\nu}, \|\cdot\|_{\widetilde{E}_i})}
\]
be an $\eta$-net for $B_{(\mathbb{R}_{4\nu}, \|\cdot\|_{\widetilde{E}_i})}$.

For each $\lambda^d = (\lambda_k^d)_{|k| \leq 4\nu}$, we define the positive numbers
\[
\varphi_k^j = \varphi_{k,\lambda^d}^j := \left( \frac{\eta}{\|e_k\|_{E_1}} + |\lambda_k^d| \right) 2^{-kj}, \quad\, j = 0, 1.
\]
Similarly, for each $\mu_i^z = (\mu_{i,k}^z)_{|k| \leq 4\nu} \in \Lambda_i$, with $2 \leq i \leq n$, we define the positive numbers
\[
\psi_{i,k}^j = \psi_{i,k,\mu_i^z}^j := \left( \frac{\eta}{\|e_k\|_{E_i}} + |\mu_{i,k}^z| \right) 2^{-kj}, \quad\, j = 0, 1.
\]

Now, let $\sigma_j > \widetilde{\beta}_j$, $j = 0,1$, and choose $N_0 \in \mathbb{Z}$ such that
\[
2^{(n-1)N_0} \leq \frac{\sigma_1}{\sigma_0} < 2^{(n-1)(N_0+1)}\,.
\]
Thus, there are finite sets $\Delta_0 = \{h_t\}_{t = 1}^{\nu_0} \subset Y^0$ and $\Delta_1 = \{f_r\}_{r = 1}^{\nu_1} \subset Y^1$ such that
\[
T(B_{X_1^0} \times \cdots \times B_{X_n^0}) \subset \bigcup_{t=1}^{\nu_0} (h_t + \sigma_0 B_{Y^0})
\]
and
\[
T(B_{X_1^1} \times \cdots \times B_{X_n^1}) \subset \bigcup_{r=1}^{\nu_1} (f_r + \sigma_1 B_{Y^1})\,.
\]

Take $\lambda^d \in \Lambda_1$, $\mu_i^z \in \Lambda_i$, $2 \leq i \leq n$, $h_t \in \Delta_0$, and $f_r \in \Delta_1$. Let
\[
-4\nu \leq m \leq 4\nu, \qquad -4\nu - N_0 \leq s_i \leq 4\nu - N_0, \quad\, 2 \leq i \leq n.
\]
Define the set
\begin{equation}\label{3.8}
S := \big(\varphi_m^0 \psi_{2,s_2+N_0}^0 \cdots \psi_{n,s_n+N_0}^0 (h_t + \sigma_0 B_{Y^0})\big)
\cap
\big(\varphi_m^1 \psi_{2,s_2+N_0}^1 \cdots \psi_{n,s_n+N_0}^1 (f_r + \sigma_1 B_{Y^1})\big)\,.
\end{equation}

Take $g_{m,s_2,\ldots,s_n} \in S$. If $S = \emptyset$, then set $g_{m,s_2,\ldots,s_n} := 0$. Define the family $(\overline{g}_{m,s_2,\ldots,s_n})$ by
\[
\overline{g}_{m,s_2,\ldots,s_n} := \begin{cases}
g_{m,s_2,\ldots,s_n}, & \text{if } |m| \leq 4\nu \,\, \text{and } \,\, |s_i+N_0| \leq 4\nu, \quad\, 2 \leq i \leq n, \\[1mm]
0, & \text{otherwise}.
\end{cases}
\]
Given $k \in \mathbb{Z}$, let
\[
\xi_k := \sum_{m=-\infty}^{\infty} \sum_{s_2=-\infty}^{\infty} \cdots \sum_{s_{n-1}=-\infty}^{\infty}
\overline{g}_{m,s_2,\ldots,s_{n-1},\,k-m-s_2-\cdots-s_{n-1}}\,.
\]
Since the family $(\overline{g}_{m,s_2,\ldots,s_n})$ has only finitely many nonzero elements, it follows that each $\xi_k$ is a finite sum in $Y^0 \cap Y^1$, and moreover $\xi_k = 0$ whenever $k \notin [-4n\nu -(n-1)N_0, 4n\nu -(n-1)N_0]$. Thus, letting
\[
\xi := \sum_{k=-\infty}^{\infty} \xi_k\,,
\]
we obtain $\xi \in Y^0 \cap Y^1 \subset J_E(\yo)$. For each choice of $\lambda^d \in \Lambda_1$, $\mu_i^z \in \Lambda_i$, $2 \leq i \leq n$, $h_t \in \Delta_0$, and $f_r \in \Delta_1$, we define in this way an element $\xi \in Y^0 \cap Y^1$.

Let $\Psi$ be the set of all elements $\xi \in Y^0 \cap Y^1$ obtained in this way. We note that $\Psi$ is finite, since the sets $\Lambda_1, \Lambda_i, \Delta_0$, and $\Delta_1$ are finite, and the elements of $\Lambda_1$ and $\Lambda_i$ are finite sequences.

Next, recall that $F_m^1 = (X_1^0 \cap X_1^1, J(2^m,\cdot,X_1^0,X_1^1))$ and 
\[
E_1(F_m^1) = \{u=(u_m)_m : \, (\|u_m\|_{F_m^1})_m \in E_1\},
\]
where
\[
\|u\|_{E_1(F_m^1)} = \|(\|u_m\|_{F_m^1})_m\|_{E_1}\,.
\]
Given $u = (u_m) \in E_1(F_m^1)$, define $\widetilde{u} = (\widetilde{u}_m)_m$ by
\[
\widetilde{u}_m := \begin{cases}
\|u_m\|_{F_m^1}, & \text{if } |m| \leq 4\nu\,, \\[1mm]
0, & \text{otherwise}.
\end{cases}
\]
Since $|\widetilde{u}_m| \leq \|u_m\|_{F_m^1}$, it follows that $\widetilde{u} \in E_1$ and
\[
\|\widetilde{u}\|_{E_1} \leq \|(\|u_m\|_{F_m^1})_m\|_{E_1}\,.
\]
In particular, we may regard $\widetilde{u}$ as an element of $(\mathbb{R}_{4\nu},\|\cdot\|_{\widetilde{E}_1})$. 
If $u = (u_m) \in B_{E_1(F_m^1)}$, then  $\|\widetilde{u}\|_{(\mathbb{R}_{4\nu},\|\cdot\|_{\widetilde{E}_1})} \leq 1$, that is,   
\[
\|(J(2^m,u_m))_{|m|\leq 4\nu}\|_{\widetilde{E}_1} \leq 1\,.
\]
Thus there exists $\lambda^d \in \Lambda_1$, where $\lambda^d = (\lambda_k^d)_{|k|\leq 4\nu}$, such that 
$\|\widetilde{u} - \lambda^d\|_{\widetilde{E}_1} \leq \eta$. This yields
\[
\|(J(2^m,u_m) - \lambda_m^d)_m\|_{\widetilde{E}_1} \leq \eta\,,
\]
or equivalently,
\[
\Big\|\sum_{|k| \leq 4\nu} (J(2^k,u_k) - \lambda_k^d)e_k\Big\|_{\widetilde{E}_1} \leq \eta\,.
\]
Hence, for $|k| \leq 4\nu$, we have
\[
|J(2^k,u_k) - \lambda_k^d| \, \|e_k\|_{E_1}
\leq
\|(J(2^m,u_m) - \lambda_m^d)_m\|_{\widetilde{E}_1}
\leq
\eta\,.
\]
Therefore, $|J(2^k,u_k) - \lambda_k^d| \leq \frac{\eta}{\|e_k\|_{E_1}}$, which implies that
\[
J(2^k,u_k) \leq \frac{\eta}{\|e_k\|_{E_1}} + |\lambda_k^d|\,.
\]

In the same way, given $v_i = (v_{i,m}) \in B_{E_i(F_m^i)}$, $2 \leq i \leq n-1$, and $v_n = (v_{n,m}) \in B_{E_n(G_m)}$, there exist $\mu_i^z 
= (\mu_{i,k}^z)_{|k|\leq 4\nu} \in \Lambda_i$, $2 \leq i \leq n$, such that
\[
|J(2^k,v_{i,k}) - \mu_{i,k}^z| \, \|e_k\|_{E_i} \leq \|(J(2^m,v_{i,m}) - \mu_{i,m}^z)_m\|_{\widetilde{E}_i} \leq \eta\,,
\]
and hence
\[
J(2^k,v_{i,k}) \leq \frac{\eta}{\|e_k\|_{E_i}} + |\mu_{i,k}^z|, \quad\, 2 \leq i \leq n\,.
\]

From the definition of $J$, we get that $\|u_k\|_{X_1^j} \leq \varphi_k^j$ and $\|v_{i,s_i+N_0}\|_{X_i^j} \leq \psi_{i,s_i+N_0}^j$, for $j = 0, 1$, $2 \leq i \leq n$, $|k| \leq 4\nu$. 
This implies that $u_k \in \varphi_k^0 B_{X_1^0} \cap \varphi_k^1 B_{X_1^1}$ and $v_{i,s_i+N_0} \in \psi_{i,s_i+N_0}^0 B_{X_i^0} \cap \psi_{i,s_i+N_0}^1 B_{X_i^1}$, for 
$-4\nu \leq k \leq 4\nu$, $-4\nu - N_0 \leq s_i \leq 4\nu - N_0$, $2 \leq i \leq n$. Then there exists $h_t \in \Delta_0$ such that
\[
T\big((\varphi_k^0)^{-1} u_k,(\psi_{2,s_2+N_0}^0)^{-1} v_{2,s_2+N_0},\ldots,(\psi_{n,s_n+N_0}^0)^{-1} v_{n,s_n+N_0}\big) \in h_t + \sigma_0 B_{Y^0}\,,
\]
which implies that
\[
T(u_k,v_{2,s_2+N_0},\ldots,v_{n,s_n+N_0}) \in \varphi_k^0 \psi_{2,s_2+N_0}^0 \cdots \psi_{n,s_n+N_0}^0 h_t + \varphi_k^0 \psi_{2,s_2+N_0}^0 \cdots \psi_{n,s_n+N_0}^0 \sigma_0 B_{Y^0}\,.
\]
Hence,
\begin{align*}
& J(2^k,T_k(u,v_2,\ldots,v_n) - \xi_k) \\
&= J\Big(2^k,\sum_{m=-\infty}^{\infty} \sum_{s_2=-\infty}^{\infty} \cdots \sum_{s_{n-1}=-\infty}^{\infty}
T(\overline{u}_m,\overline{v}_{2,s_2+N_0},\ldots,\overline{v}_{n-1,s_{n-1}+N_0},\overline{v}_{n,k-m-s_2-\cdots-s_{n-1}+N_0}) \\
&\quad\, - \sum_{m=-\infty}^{\infty} \sum_{s_2=-\infty}^{\infty} \cdots \sum_{s_{n-1}=-\infty}^{\infty}
\overline{g}_{m,s_2,\ldots,s_{n-1},\,k-m-s_2-\cdots-s_{n-1}}\Big) \\
&\leq \bigg(\sum_{m=-\infty}^{\infty} \sum_{s_2=-\infty}^{\infty} \cdots \sum_{s_{n-1}=-\infty}^{\infty}
\Big(J\big(2^k,T(\overline{u}_m,\overline{v}_{2,s_2+N_0},\ldots,\overline{v}_{n-1,s_{n-1}+N_0},\overline{v}_{n,k-m-s_2-\cdots-s_{n-1}+N_0}) \\
&\qquad\qquad\qquad\qquad\qquad\qquad
- \overline{g}_{m,s_2,\ldots,s_{n-1},\,k-m-s_2-\cdots-s_{n-1}}\big)\Big)^r \bigg)^{1/r} \\
&= \bigg(\sum_{m=-\infty}^{\infty} \sum_{s_2=-\infty}^{\infty} \cdots \sum_{s_{n-1}=-\infty}^{\infty}
\max \Big\{
\|T(\overline{u}_m,\overline{v}_{2,s_2+N_0},\ldots,\overline{v}_{n-1,s_{n-1}+N_0},\overline{v}_{n,k-m-s_2-\cdots-s_{n-1}+N_0}) \\
&\qquad\qquad\qquad\qquad\qquad\qquad
- \overline{g}_{m,s_2,\ldots,s_{n-1},\,k-m-s_2-\cdots-s_{n-1}}\|_{Y^0}, \\
&\qquad\qquad\qquad\qquad\qquad\qquad
2^k \|T(\overline{u}_m,\overline{v}_{2,s_2+N_0},\ldots,\overline{v}_{n-1,s_{n-1}+N_0},\overline{v}_{n,k-m-s_2-\cdots-s_{n-1}+N_0}) \\
&\qquad\qquad\qquad\qquad\qquad\qquad
- \overline{g}_{m,s_2,\ldots,s_{n-1},\,k-m-s_2-\cdots-s_{n-1}}\|_{Y^1}
\Big\}^r \bigg)^{1/r} \\
&\leq \bigg(\sum_{m=-\infty}^{\infty} \sum_{s_2=-\infty}^{\infty} \cdots \sum_{s_{n-1}=-\infty}^{\infty}
\max \Big\{
2 \varphi_m^0 \psi_{2,s_2+N_0}^0 \cdots \psi_{n-1,s_{n-1}+N_0}^0 \psi_{n,k-m-s_2-\cdots-s_{n-1}+N_0}^0 \sigma_0, \\
&\qquad\qquad\qquad\qquad\qquad
2 \cdot 2^k \varphi_m^1 \psi_{2,s_2+N_0}^1 \cdots \psi_{n-1,s_{n-1}+N_0}^1 \psi_{n,k-m-s_2-\cdots-s_{n-1}+N_0}^1 \sigma_1
\Big\}^r \bigg)^{1/r} \\
&=^{**} \bigg(\sum_{m=-\infty}^{\infty} \sum_{s_2=-\infty}^{\infty} \cdots \sum_{s_{n-1}=-\infty}^{\infty}
\max \Big\{
2 \varphi_m^0 \psi_{2,s_2+N_0}^0 \cdots \psi_{n-1,s_{n-1}+N_0}^0 \psi_{n,k-m-s_2-\cdots-s_{n-1}+N_0}^0 \sigma_0, \\
&\qquad\qquad\qquad\qquad\qquad
2 \cdot 2^k \cdot 2^{-m} \varphi_m^0 \cdot 2^{-s_2-N_0} \psi_{2,s_2+N_0}^0 \cdots 2^{-s_{n-1}-N_0} \psi_{n-1,s_{n-1}+N_0}^0 \\
&\qquad\qquad\qquad\qquad\qquad
\times 2^{-k+m+s_2+\cdots+s_{n-1}-N_0}
\psi_{n,k-m-s_2-\cdots-s_{n-1}+N_0}^0 \sigma_1
\Big\}^r \bigg)^{1/r} \\
&= \bigg(\sum_{m=-\infty}^{\infty} \sum_{s_2=-\infty}^{\infty} \cdots \sum_{s_{n-1}=-\infty}^{\infty}
\max \Big\{
2 \sigma_0 \varphi_m^0 \psi_{2,s_2+N_0}^0 \cdots \psi_{n-1,s_{n-1}+N_0}^0 \psi_{n,k-m-s_2-\cdots-s_{n-1}+N_0}^0, \\
&\qquad\qquad\qquad\qquad\qquad
2 \cdot 2^{-(n-1)N_0} \sigma_1
\varphi_m^0 \psi_{2,s_2+N_0}^0 \cdots \psi_{n-1,s_{n-1}+N_0}^0
\psi_{n,k-m-s_2-\cdots-s_{n-1}+N_0}^0
\Big\}^r \bigg)^{1/r} \\
&\leq \bigg(\sum_{m=-\infty}^{\infty} \sum_{s_2=-\infty}^{\infty} \cdots \sum_{s_{n-1}=-\infty}^{\infty}
\max \Big\{
2 \sigma_0 \varphi_m^0 \psi_{2,s_2+N_0}^0 \cdots \psi_{n-1,s_{n-1}+N_0}^0 \psi_{n,k-m-s_2-\cdots-s_{n-1}+N_0}^0, \\
&\qquad\qquad\qquad\qquad\qquad
2^n \sigma_0
\varphi_m^0 \psi_{2,s_2+N_0}^0 \cdots \psi_{n-1,s_{n-1}+N_0}^0
\psi_{n,k-m-s_2-\cdots-s_{n-1}+N_0}^0
\Big\}^r \bigg)^{1/r} \\
&\leq 2^n \sigma_0 \bigg(\sum_{m=-\infty}^{\infty} \sum_{s_2=-\infty}^{\infty} \cdots \sum_{s_{n-1}=-\infty}^{\infty}
\Big(\varphi_m^0 \psi_{2,s_2+N_0}^0 \cdots \psi_{n-1,s_{n-1}+N_0}^0
\psi_{n,k-m-s_2-\cdots-s_{n-1}+N_0}^0\Big)^r \bigg)^{1/r}.
\end{align*}
We remark that in $=^{**}$ above we used that $\varphi_m^1 = 2^{-m}\varphi_m^0$ and
\[
\psi_{i,s_i+N_0}^1 = 2^{-s_i-N_0}\psi_{i,s_i+N_0}^0, \quad\, 2 \leq i \leq n-1,
\]
as well as
\[
\psi_{n,k-m-s_2-\cdots-s_{n-1}+N_0}^1
=
2^{-k+m+s_2+\cdots+s_{n-1}-N_0}
\psi_{n,k-m-s_2-\cdots-s_{n-1}+N_0}^0\,.
\]
Here we have used that
\[
2^{-(n-1)N_0}\sigma_1 < 2^{n-1}\sigma_0\,.
\]
Since $\xi = \sum_{k=-\infty}^{\infty} \xi_k$, with $\xi_k \in Y^0 \cap Y^1$, and
\[
T(\pi_1 R_{1,4\nu} u,\pi_2 R_{2,4\nu} v_2,\ldots,\pi_{n-1} R_{n-1,4\nu} v_{n-1},\pi_n S_{4\nu} v_n)
=
\sum_k T_k(u,v_2,\ldots,v_n)\,,
\]
where $T_k(u,v_2,\ldots,v_n) \in Y^0 \cap Y^1$, we have
\[
T(\pi_1 R_{1,4\nu} u,\pi_2 R_{2,4\nu} v_2,\ldots,\pi_{n-1} R_{n-1,4\nu} v_{n-1},\pi_n S_{4\nu} v_n) - \xi
=
\sum_k \big(T_k(u,v_2,\ldots,v_n) - \xi_k\big)\,.
\]

\begin{align*}
& \|T(\pi_1 R_{1,4\nu} u, \pi_2 R_{2,4\nu} v_2,\ldots,\pi_{n-1} R_{n-1,4\nu} v_{n-1}, \pi_n S_{4\nu} v_n) - \xi\|_{J_E(\yo)} \\
&\leq \|(J(2^k,T_k(u,v_2,\ldots,v_n) - \xi_k))_k\|_E \\
&\leq 2^n \sigma_0 \Big\|\Big(\sum_{m=-\infty}^{\infty} \sum_{s_2=-\infty}^{\infty} \cdots \sum_{s_{n-1}=-\infty}^{\infty}
\big(\varphi_m^0 \psi_{2,s_2+N_0}^0 \cdots \psi_{n-1,s_{n-1}+N_0}^0 \psi_{n,k-m-s_2-\cdots-s_{n-1}+N_0}^0\big)^r\Big)^{1/r}\Big\|_E \\
&\leq 2^n \sigma_0 \gamma \|(\varphi_m^0)\|_{E_1}\|\tau_{N_0}((\psi_{2,s_2}^0))\|_{E_2}\cdots \|\tau_{N_0}((\psi_{n,s_n}^0))\|_{E_n}\,. \qquad (++)
\end{align*}

Now, since
\[
\varphi_k^0 = \frac{\eta}{\|e_k\|_{E_1}} + |\lambda_k^d|\,,
\]
where $\lambda^d = (\lambda_k^d) \in B_{(\mathbb{R}_{4\nu}, \|\cdot\|_{\widetilde{E}_1})}$, we have
\begin{align*}
\|(\varphi_m^0)\|_{E_1}
&=
\big\|(\eta /\|e_k\|_{E_1} + |\lambda_k^d|)_{|k| \leq 4\nu}\big\|_{E_1} \\
&\leq
c_{E_1} \Big(
\Big\|\sum_{|k|\leq 4\nu} \frac{\eta}{\|e_k\|_{E_1}} e_k \Big\|_{\widetilde{E}_1}
+
\Big\|\sum_{|k|\leq 4\nu} |\lambda_k^d| e_k \Big\|_{\widetilde{E}_1}
\Big) \\
&\leq
c_{E_1}(1+1)
=
2c_{E_1}\,.
\end{align*}
We also have, for $2 \leq i \leq n$, that
\[
\|(\psi_{i,s_i}^0)\|_{E_i} \leq 2c_{E_i}\,.
\]

Consequently, for each $2 \leq i \leq n$, we obtain
\begin{align*}
\|\tau_{N_0}((\psi_{i,s_i}^0))\|_{E_i}
&\leq
\|\tau_{N_0}\|_{E_i \to E_i}\|(\psi_{i,s_i}^0)\|_{E_i} \\
&\leq
2c_{E_i}\|\tau_{N_0}\|_{E_i \to E_i}\,.
\end{align*}
Since, for each $2 \leq i \leq n$,
\[
\|\tau_{N_0}\|_{E_i \to E_i}
=
\varphi_{E_i}(2^{N_0})
\leq
\varphi_{E_i}\bigg(\bigg(\frac{\sigma_1}{\sigma_0}\bigg)^{1/(n-1)}\bigg),
\]
we obtain
\begin{align*}
(++)
&\leq
2^n \sigma_0 M \, 2c_{E_1}\, 2c_{E_2} \cdots 2c_{E_n}
\, \varphi_{E_2}(2^{N_0}) \cdots \varphi_{E_n}(2^{N_0}) \\
&\leq
2^{2n} c_{E_1} c_{E_2} \cdots c_{E_n} M \sigma_0
\varphi_{E_2}\bigg(\bigg(\frac{\sigma_1}{\sigma_0}\bigg)^{1/(n-1)}\bigg)\cdots
\varphi_{E_n}\bigg(\bigg(\frac{\sigma_1}{\sigma_0}\bigg)^{1/(n-1)}\bigg) \\
&=
\widetilde{C}\,\sigma_0
\varphi_{E_2}\bigg(\bigg(\frac{\sigma_1}{\sigma_0}\bigg)^{1/(n-1)}\bigg)\cdots
\varphi_{E_n}\bigg(\bigg(\frac{\sigma_1}{\sigma_0}\bigg)^{1/(n-1)}\bigg)\,.
\end{align*}

Now, if we let 
\[
A_\nu := P_{(n+1)\nu} \widehat{T}(R_{1,4\nu},R_{2,4\nu},\ldots,R_{n-1,4\nu},S_{4\nu})\,.  
\]
Then, we obtain.
\begin{align*}
\widetilde{\beta}\bigg(A_\nu \colon \bigg(\prod_{i=1}^{n-1} E_i(F_m^i)\bigg) & \times E_n(G_m) \to E(W_m)\bigg) \\
&\leq \widetilde{C}(1+\varepsilon)\,\widetilde{\beta}_0 
\varphi_{E_2}\bigg(\bigg(\frac{\widetilde{\beta}_1}{\widetilde{\beta}_0}\bigg)^{1/(n-1)}\bigg)\cdots
\varphi_{E_n}\bigg(\bigg(\frac{\widetilde{\beta}_1}{\widetilde{\beta}_0}\bigg)^{1/(n-1)}\bigg)\,.
\end{align*}
Since $\varepsilon > 0$ is arbitrary, it follows that
\begin{align*}
\widetilde{\beta}\bigg(A_\nu \colon \bigg(\prod_{i=1}^{n-1} E_i(F_m^i)\bigg) & \times E_n(G_m) \to E(W_m)\bigg) \\
&\leq \widetilde{C}\,\widetilde{\beta}_0  \varphi_{E_2}\bigg(\bigg(\frac{\widetilde{\beta}_1}{\widetilde{\beta}_0}\bigg)^{1/(n-1)}\bigg)\cdots
\varphi_{E_n}\bigg(\bigg(\frac{\widetilde{\beta}_1}{\widetilde{\beta}_0}\bigg)^{1/(n-1)}\bigg)\,.
\end{align*}

{\bf Step 3}. It remains to estimate the remaining terms in the decomposition. We treat them by combining endpoint bounds with Lemma \ref{lemma1} and Theorem \ref{conv}.

As
\[
T\colon (X_1^0 + X_1^1) \times (X_2^0 + X_2^1) \times \cdots \times (X_n^0 + X_n^1) \rightarrow Y^0 + Y^1
\]
is bounded, it follows that
\[
T\colon X_1^{j_1} \times X_2^{j_2} \times \cdots \times X_n^{j_n} \rightarrow Y^0 + Y^1
\]
is also bounded, where $j_i = 0$ or $j_i = 1$, for $1 \leq i \leq n$.
{\bf a)} Let
\[
T_{1,\nu} := P_{(n+1)\nu}\widehat{T}(R_{1,4\nu},\ldots,R_{n-1,4\nu},S_{4\nu}^+).
\]
Consider the diagram
\[
\begin{tikzcd}
\ell_p(F_m^1) \times \cdots \times \ell_p(F_m^{n-1}) \times \ell_p(G_m) \arrow{r}{T_{1,\nu}} \arrow[swap]{d}{(R_{1,4\nu},\ldots,R_{n-1,4\nu},S_{4\nu}^+)} & \ell_\infty(W_m) \\
\ell_p(F_m^1) \times \cdots \times \ell_p(F_m^{n-1}) \times \ell_p(2^{-m}G_m) \arrow{r}{\widehat{T}} & \ell_\infty(W_m) + \ell_\infty(2^{-m}W_m) \arrow[swap]{u}{P_{(n+1)\nu}}
\end{tikzcd}
\]
From \eqref{eq1}, we have
\[
\|R_{i,\nu}\|_{L(\ell_p(F_m^i) + \ell_p(2^{-m} F_m^i), \ell_p(F_m^i) \cap \ell_p(2^{-m} F_m^i))} \leq c_{\overrightarrow{X}_i} 2^\nu\,,
\]
and from \eqref{eq2},
\[
\|R_\nu^+\|_{L(\ell_p(F_m^i), \ell_p(2^{-m} F_m^i))} \leq 2^{-(\nu+1)}, \qquad \|R_\nu^-\|_{L(\ell_p(2^{-m} F_m^i), \ell_p(F_m^i))} \leq 2^{-(\nu+1)}.
\]
Moreover,
\[
\|S_{4\nu} + S_{4\nu}^+\|_{L(\ell_p(G_m), \ell_p(G_m))} \leq 1, \qquad \|S_{4\nu} + S_{4\nu}^-\|_{L(\ell_p(2^{-m}G_m), \ell_p(2^{-m} G_m))} \leq 1,
\]
and
\[
\|P_{(n+1)\nu}\|_{L(\ell_\infty(W_m) + \ell_\infty(2^{-m} W_m),\ell_\infty(W_m))} \leq c_{\overrightarrow{Y}} 2^{3\nu}.
\]
Thus,
\begin{align*}
& \|P_{(n+1)\nu}\widehat{T}(R_{1,4\nu},\ldots,R_{n-1,4\nu},S_{4\nu}^+)\|_{L(\ell_p(F_m^1),\ldots,\ell_p(F_m^{n-1}),\ell_p(G_m);\ell_\infty(W_m))} \\
& \leq \|P_{(n+1)\nu}\|_{L(\ell_\infty(W_m) + \ell_\infty(2^{-m}W_m),\ell_\infty(W_m))} \\
& \quad\, \times \|\widehat{T}\|_{L(\ell_p(F_m^1),\ldots,\ell_p(F_m^{n-1}),\ell_p(2^{-m}G_m);\ell_\infty(W_m) + \ell_\infty(2^{-m}W_m))} \\
& \quad\, \times \|R_{1,4\nu}\|_{L(\ell_p(F_m^1),\ell_p(F_m^1))} \cdots \|R_{n-1,4\nu}\|_{L(\ell_p(F_m^{n-1}),\ell_p(F_m^{n-1}))} \|S_{4\nu}^+\|_{L(\ell_p(G_m),\ell_p(2^{-m} G_m))} \\
& \leq c_{\overrightarrow{Y}} 2^{3\nu} \|\widehat{T}\|_{L(\ell_p(F_m^1),\ldots,\ell_p(F_m^{n-1}),\ell_p(2^{-m}G_m);\ell_\infty(W_m) + \ell_\infty(2^{-m}W_m))} \cdot 1 \cdots 1 \cdot 2^{-(4\nu+1)} \\
& \leq c_{\overrightarrow{Y}} 2^{-\nu} \|\widehat{T}\|_{L(\ell_p(F_m^1),\ldots,\ell_p(F_m^{n-1}),\ell_p(2^{-m}G_m);\ell_\infty(W_m) + \ell_\infty(2^{-m}W_m))}\,.
\end{align*}

Now consider the diagram
\begin{align*}
\ell_p(F_m^1) \times \cdots \times \ell_p(F_m^{n-1}) \times \ell_p(2^{-m}G_m) & \stackrel{(\pi_1,\ldots,\pi_n)}{\longrightarrow} X_1^0 \times \cdots \times X_{n-1}^0 \times X_n^1 \\
& \stackrel{T}{\longrightarrow} Y^0 + Y^1 \stackrel{Q}{\longrightarrow} \ell_\infty(W_m) + \ell_\infty(2^{-m}W_m)\,.
\end{align*}
Since
\begin{align*}
& \|\widehat{T}\|_{L(\ell_p(F_m^1), \ldots, \ell_p(F_m^{n-1}), \ell_p(2^{-m}G_m); \ell_\infty(W_m) + \ell_\infty(2^{-m}W_m))} \\
& \leq \|(\pi_1,\ldots,\pi_n)\| \|T\|_{L(X_1^0, \ldots, X_{n-1}^0, X_n^1; Y^0 + Y^1)} \|Q\| \\
& \leq \|T\|_{L(X_1^0, \ldots, X_{n-1}^0, X_n^1; Y^0 + Y^1)}\,,
\end{align*}
we get the estimate
\begin{align*}
& \|P_{(n+1)\nu}\widehat{T}(R_{1,4\nu}, \ldots, R_{n-1,4\nu}, S_{4\nu}^{+})\|_{L(\ell_p(F_m^1), \ldots, \ell_p(F_m^{n-1}), \ell_p(G_m); \ell_\infty(W_m))} \\
&\leq c_{\overrightarrow{Y}} 2^{-\nu} \|T\|_{L(X_1^0, \ldots, X_{n-1}^0, X_n^1; Y^0 + Y^1)}
\rightarrow 0 \quad \text{as \, $\nu \to \infty$}\,.
\end{align*}
Combining this estimate with
\[
\|P_{(n+1)\nu}\widehat{T}(R_{1,4\nu}, \ldots, R_{n-1,4\nu}, S_{4\nu}^+)\|_{L(\ell_p(2^{-m}F_m^1), \ldots, \ell_p(2^{-m}F_m^{n-1}), \ell_p(2^{-m}G_m); \ell_\infty(2^{-m}W_m))}
\leq \|T\|_1\,,
\]
we define
\[
a_\nu := c_{\overrightarrow{Y}} 2^{-\nu} \|T\|_{L(X_1^0, \ldots, X_{n-1}^0, X_n^1; Y^0 + Y^1)}.
\]
Then $T_{1,\nu}$ is an $n$-linear operator from
\[
\vec F_1 \times \cdots \times \vec F_{n-1} \times \vec G
\quad \text{to} \quad
\vec W_\infty,
\]
where
\[
\vec F_i := (\ell_p(F_m^i), \ell_p(2^{-m}F_m^i)), \quad 1 \leq i \leq n-1,
\]
\[
\vec G := (\ell_p(G_m), \ell_p(2^{-m}G_m)),
\qquad
\vec W_\infty := (\ell_\infty(W_m), \ell_\infty(2^{-m}W_m)).
\]
Moreover,
\[
\|T_{1,\nu}\|_0 \leq a_\nu,
\qquad
\|T_{1,\nu}\|_1 \leq \|T\|_1.
\]
Hence, by Lemma \ref{lemma1} and Theorem \ref{conv},
\begin{align*}
& \widetilde{\beta}\bigg(T_{1,\nu} \colon \bigg(\prod_{i=1}^{n-1} E_i(F_m^i)\bigg) \times E_n(G_m) \to E(W_m)\bigg) \\
&\leq \|T_{1,\nu}\|_{L((\prod_{i=1}^{n-1} E_i(F_m^i)) \times E_n(G_m); E(W_m))} \\
&\leq C a_\nu \,
\varphi_{E_2}\bigg(\bigg(\frac{\|T\|_1}{a_\nu}\bigg)^{1/(n-1)}\bigg)\cdots
\varphi_{E_n}\bigg(\bigg(\frac{\|T\|_1}{a_\nu}\bigg)^{1/(n-1)}\bigg).
\end{align*}
If $\|T\|_1 = 0$, then the conclusion is immediate. Assume that $\|T\|_1 > 0$ and set
\[
t_\nu := \bigg(\frac{\|T\|_1}{a_\nu}\bigg)^{1/(n-1)}.
\]
Then $t_\nu \to \infty$ as $\nu \to \infty$, and
\[
a_\nu \,
\varphi_{E_2}(t_\nu)\cdots \varphi_{E_n}(t_\nu)
=
\|T\|_1
\frac{\varphi_{E_2}(t_\nu)}{t_\nu}\cdots
\frac{\varphi_{E_n}(t_\nu)}{t_\nu}.
\]
Since each $E_i$, $2 \leq i \leq n$, satisfies condition $(2)$, the right-hand side tends to $0$ as $\nu \to \infty$. Consequently,
\[
\widetilde{\beta}\bigg(T_{1,\nu} \colon \bigg(\prod_{i=1}^{n-1} E_i(F_m^i)\bigg) \times E_n(G_m) \to E(W_m)\bigg)
\to 0 \quad \text{as \, $\nu \to \infty$}\,.
\]
{\bf b)} For
\[
T_{2,\nu} := P_{(n+1)\nu}\widehat{T}(R_{1,4\nu},\ldots,R_{n-1,4\nu},S_{4\nu}^-),
\]
we consider the diagram
\[
\begin{tikzcd}
\ell_p(2^{-m} F_m^1) \times \cdots \times \ell_p(2^{-m}F_m^{n-1}) \times \ell_p(2^{-m} G_m) \arrow{r}{T_{2,\nu}} \arrow[swap]{d}{(R_{1,4\nu},\ldots,R_{n-1,4\nu},S_{4\nu}^-)} & \ell_\infty(2^{-m} W_m) \\
\ell_p(2^{-m} F_m^1) \times \cdots \times \ell_p(2^{-m} F_m^{n-1}) \times \ell_p(G_m) \arrow{r}{\widehat{T}} & \ell_\infty(W_m) + \ell_\infty(2^{-m}W_m) \arrow[swap]{u}{P_{(n+1)\nu}}
\end{tikzcd}
\]
In the same way as in {\bf Step 3(a)}, we obtain
\[
\widetilde{\beta}\bigg(T_{2,\nu} \colon \bigg(\prod_{i=1}^{n-1} E_i(F_m^i)\bigg) \times E_n(G_m) \to E(W_m)\bigg)
\to 0 \quad \text{as \, $\nu \to \infty$}\,.
\]

\textbf{Step 4}. Here we consider case {\rm(ii)} of the decomposition \eqref{dec1}, that is, the terms in which, in the first $n-1$ variables, 
all operators are of the form $R_{i,4\nu}^{+}$.
\smallskip

\textbf{a)} For
\[
T_{3,\nu} := P_{(n+1)\nu}\widehat{T}(R_{1,4\nu}^{+}, \ldots, R_{n-1,4\nu}^{+}, S_{4\nu} + S_{4\nu}^{+}),
\]
we consider the diagram
\[
\begin{tikzcd}
\ell_p(F_m^1) \times \ell_p(F_m^2) \times \cdots \times \ell_p(F_m^{n-1}) \times \ell_p(G_m)
\arrow{r}{T_{3,\nu}}
\arrow[swap]{d}{(R_{1,4\nu}^{+}, R_{2,4\nu}^{+}, \ldots, R_{n-1,4\nu}^{+}, S_{4\nu} + S_{4\nu}^{+})}
& \ell_\infty(W_m) \\
\ell_p(2^{-m} F_m^1) \times \ell_p(2^{-m} F_m^2) \times \cdots \times \ell_p(2^{-m} F_m^{n-1}) \times \ell_p(G_m)
\arrow{r}{\widehat{T}}
& \ell_\infty(W_m) + \ell_\infty(2^{-m}W_m)
\arrow[swap]{u}{P_{(n+1)\nu}}
\end{tikzcd}
\]
In the same way as in {\bf Step 3(a)}, we obtain
\[
\widetilde{\beta}\bigg(T_{3,\nu} \colon \bigg(\prod_{i=1}^{n-1} E_i(F_m^i)\bigg) \times E_n(G_m) \to E(W_m)\bigg)
\to 0
\quad \text{as \, $\nu \to \infty$}\,.
\]

\smallskip

\textbf{b)} For
\[
T_{4,\nu} := P_{(n+1)\nu}\widehat{T}(R_{1,4\nu}^{+}, \ldots, R_{n-1,4\nu}^{+}, S_{4\nu}^{-}),
\]
we consider the diagram
\[
\begin{tikzcd}
\ell_p(F_m^1) \times \ell_p(F_m^2) \times \cdots \times \ell_p(F_m^{n-1}) \times \ell_p(2^{-m} G_m)
\arrow{r}{T_{4,\nu}}
\arrow[swap]{d}{(R_{1,4\nu}^{+}, R_{2,4\nu}^{+}, \ldots, R_{n-1,4\nu}^{+}, S_{4\nu}^{-})}
& \ell_\infty(2^{-m}W_m) \\
\ell_p(2^{-m} F_m^1) \times \ell_p(2^{-m} F_m^2) \times \cdots \times \ell_p(2^{-m} F_m^{n-1}) \times \ell_p(G_m)
\arrow{r}{\widehat{T}}
& \ell_\infty(W_m) + \ell_\infty(2^{-m}W_m)
\arrow[swap]{u}{P_{(n+1)\nu}}
\end{tikzcd}
\]
In the same way as in {\bf Step 3(b)}, we obtain
\[
\widetilde{\beta}\bigg(T_{4,\nu} \colon \bigg(\prod_{i=1}^{n-1} E_i(F_m^i)\bigg) \times E_n(G_m) \to E(W_m)\bigg)
\to 0
\quad \text{as \, $\nu \to \infty$}\,.
\]

\medskip

{\bf Step 5}. Here we consider case {\rm(iii)} of the decomposition, where all operators are of the form $R_{i,4\nu}^{-}$. In this case, we argue exactly as in 
{\bf Step 4(a)} and {\bf Step 4(b)}, with $R_{i,4\nu}^{+}$ replaced by $R_{i,4\nu}^{-}$. Consequently, all corresponding terms tend to zero as $\nu \to \infty$.

\medskip

{\bf Step 6}. Here we consider case {\rm(iv)} of the decomposition, where at least one operator is of the form $R_{i,4\nu}$ and at 
least one further operator is of the form $R_{j,4\nu}^{+}$ and/or $R_{k,4\nu}^{-}$. Without loss of generality, we consider the representative term
\[
P_{(n+1)\nu}\widehat{T}(R_{1,4\nu}, R_{2,4\nu}^{+}, \ldots, S_{4\nu} + S_{4\nu}^{+} + S_{4\nu}^{-})\,.
\]
We decompose it as
\begin{align*}
P_{(n+1)\nu}\widehat{T}(R_{1,4\nu}, R_{2,4\nu}^{+}, \ldots, S_{4\nu} + S_{4\nu}^{+} + S_{4\nu}^{-})
&= P_{(n+1)\nu}\widehat{T}(R_{1,4\nu}, R_{2,4\nu}^{+}, \ldots, S_{4\nu} + S_{4\nu}^{+}) \\
&\quad + P_{(n+1)\nu}\widehat{T}(R_{1,4\nu}, R_{2,4\nu}^{+}, \ldots, S_{4\nu}^{-})\,.
\end{align*}
Set
\[
T_{5,\nu} := P_{(n+1)\nu}\widehat{T}(R_{1,4\nu}, R_{2,4\nu}^{+}, \ldots, S_{4\nu} + S_{4\nu}^{+})
\]
and
\[
T_{6,\nu} := P_{(n+1)\nu}\widehat{T}(R_{1,4\nu}, R_{2,4\nu}^{+}, \ldots, S_{4\nu}^{-})\,.
\]
The corresponding factorizations can be represented by diagrams analogous to those used in Steps 3--5, and we therefore omit them.
Then $T_{5,\nu}$ is treated exactly as in {\bf Step 3(a)} and {\bf Step 4(a)}, while $T_{6,\nu}$ is treated exactly as in {\bf Step 3(b)} and {\bf Step 4(b)}.
Consequently,
\[
\widetilde{\beta}\bigg(T_{5,\nu} \colon \bigg(\prod_{i=1}^{n-1} E_i(F_m^i)\bigg) \times E_n(G_m) \to E(W_m)\bigg)
\to 0
\quad \text{as \, $\nu \to \infty$},
\]
and
\[
\widetilde{\beta}\bigg(T_{6,\nu} \colon \bigg(\prod_{i=1}^{n-1} E_i(F_m^i)\bigg) \times E_n(G_m) \to E(W_m)\bigg)
\to 0
\quad \text{as \, $\nu \to \infty$}.
\]
The same argument applies to any term occurring in case {\rm(iv)}. Therefore, all terms in case {\rm(iv)} tend to zero as $\nu \to \infty$.

\medskip

{\bf Step 7}. Here we consider case {\rm(v)} of the decomposition, where we have only operators of the form $R_i^{+}$ and $R_j^{-}$, with 
at least one operator of each type. We argue as in {\bf Step 6}. Consequently, all terms in case {\rm(v)} tend to zero as $\nu \to \infty$.

\medskip

It remains to analyze the last two operators in the decomposition.

\medskip

{\bf Step 8}. For $P_{(n+1)\nu}^{+}\widehat{T}$, we consider the decomposition
\begin{align*}
P_{(n+1)\nu}^{+}\widehat{T}
&= P_{(n+1)\nu}^{+}\widehat{T}(I,I,\ldots,I) \\
&= P_{(n+1)\nu}^{+}\widehat{T}((R_{1,\nu}+R_{1,\nu}^{-})+R_{1,\nu}^{+}, \ldots, (R_{n-1,\nu}+R_{n-1,\nu}^{-})+R_{n-1,\nu}^{+}, I) \\
&= P_{(n+1)\nu}^{+}\widehat{T}(R_{1,\nu}+R_{1,\nu}^{-}, \ldots, R_{n-1,\nu}+R_{n-1,\nu}^{-}, I) \\
&\quad + P_{(n+1)\nu}^{+}\widehat{T}(R_{1,\nu}+R_{1,\nu}^{-}, R_{2,\nu}+R_{2,\nu}^{-}, \ldots, R_{n-1,\nu}^{+}, I) \\
&\quad + \cdots \\
&\quad + P_{(n+1)\nu}^{+}\widehat{T}(R_{1,\nu}^{+}, R_{2,\nu}^{+}, \ldots, R_{n-1,\nu}^{+}, I)\,,
\end{align*}
where we have a total of $2^{n-1}$ terms.

{\bf a)} We write
\begin{align*}
& P_{(n+1)\nu}^{+}\widehat{T}(R_{1,\nu}+R_{1,\nu}^{-}, \ldots, R_{n-1,\nu}+R_{n-1,\nu}^{-}, I) \\
&= P_{(n+1)\nu}^{+}\widehat{T}(R_{1,\nu}+R_{1,\nu}^{-}, \ldots, R_{n-1,\nu}+R_{n-1,\nu}^{-}, S_{\nu}+S_{\nu}^{-}) \\
&\quad + P_{(n+1)\nu}^{+}\widehat{T}(R_{1,\nu}+R_{1,\nu}^{-}, \ldots, R_{n-1,\nu}+R_{n-1,\nu}^{-}, S_{\nu}^{+})\,.
\end{align*}
Set
\[
T_{7,\nu}:=P_{(n+1)\nu}^{+}\widehat{T}(R_{1,\nu}+R_{1,\nu}^{-}, \ldots, R_{n-1,\nu}+R_{n-1,\nu}^{-}, S_{\nu}+S_{\nu}^{-}).
\]
Since
\[
R_{\nu}+R_{\nu}^{-}\colon \ell_p(2^{-m}F_m)\to \ell_p(F_m)
\quad\text{and}\quad
\|R_{\nu}+R_{\nu}^{-}\|_{L(\ell_p(2^{-m}F_m),\ell_p(F_m))}\le 2^{\nu},
\]
consider the diagram
\[
\begin{tikzcd}
\ell_p(2^{-m} F_m^1) \times \cdots \times \ell_p(2^{-m}F_m^{n-1}) \times \ell_p(2^{-m}G_m)
\arrow{r}{T_{7,\nu}}
\arrow[swap]{d}{(R_{1,\nu}+R_{1,\nu}^{-}, \ldots, R_{n-1,\nu}+R_{n-1,\nu}^{-}, S_{\nu}+S_{\nu}^{-})}
& \ell_\infty(2^{-m}W_m) \\
\ell_p(F_m^1) \times \cdots \times \ell_p(F_m^{n-1}) \times \ell_p(G_m)
\arrow{r}{\widehat{T}}
& \ell_\infty(W_m) \arrow[swap]{u}{P_{(n+1)\nu}^{+}}\,.
\end{tikzcd}
\]
Then
\begin{align*}
& \|T_{7,\nu}\|_{L(\ell_p(2^{-m} F_m^1), \ldots, \ell_p(2^{-m}F_m^{n-1}), \ell_p(2^{-m}G_m); \ell_\infty(2^{-m}W_m))} \\
&\leq \|P_{(n+1)\nu}^{+}\|_{L(\ell_\infty(W_m), \ell_\infty(2^{-m}W_m))}
\|\widehat{T}\|_{L(\ell_p(F_m^1), \ldots, \ell_p(F_m^{n-1}), \ell_p(G_m); \ell_\infty(W_m))} \\
&\quad \times \|R_{1,\nu}+R_{1,\nu}^{-}\|_{L(\ell_p(2^{-m}F_m^1), \ell_p(F_m^1))}\cdots
\|R_{n-1,\nu}+R_{n-1,\nu}^{-}\|_{L(\ell_p(2^{-m}F_m^{n-1}), \ell_p(F_m^{n-1}))} \\
&\quad \times \|S_{\nu}+S_{\nu}^{-}\|_{L(\ell_p(2^{-m}G_m), \ell_p(G_m))} \\
&\leq 2^{-(n+1)\nu}\|T\|_0\,2^{n\nu}
=2^{-\nu}\|T\|_0 \to 0
\quad \text{as \, $\nu\to\infty$}\,.
\end{align*}
Moreover,
\[
\|T_{7,\nu}\|_{L(\ell_p(F_m^1), \ldots, \ell_p(F_m^{n-1}), \ell_p(G_m); \ell_\infty(W_m))} \leq \|T\|_0.
\]
Therefore, arguing as in {\bf Step 3(a)}, we conclude that
\[
\widetilde{\beta}\bigg(T_{7,\nu} \colon \bigg(\prod_{i=1}^{n-1} E_i(F_m^i)\bigg) \times E_n(G_m) \to E(W_m)\bigg)
\to 0
\quad \text{as \, $\nu\to\infty$}\,.
\]

{\bf b)} Now consider operator
\[
T_{8,\nu}:=P_{(n+1)\nu}^{+}\widehat{T}(R_{1,\nu}+R_{1,\nu}^{-},\ldots,R_{n-1,\nu}+R_{n-1,\nu}^{-},S_{\nu}^{+}).
\]
For $j=0,1$, let $f_j^i$ denote the set of all finitely supported sequences in $\ell_p(2^{-mj}F_m^i)$, $1 \leq i \leq n-1$, 
and let $g_j$ denote the set of all finitely supported sequences in $\ell_p(2^{-mj}G_m)$.

Since
\[
T\colon (X_1^0+X_1^1)\times \cdots \times (X_n^0+X_n^1)\to Y^0+Y^1
\]
is bounded, it follows that
\[
T\colon X_1^{j_1}\times \cdots \times X_n^{j_n}\to Y^0+Y^1
\]
is bounded for all $j_1,\ldots,j_n\in\{0,1\}$. Hence
\begin{align*}
& \|T(\pi_1(R_{1,\nu}+R_{1,\nu}^{-}),\ldots,\pi_{n-1}(R_{n-1,\nu}+R_{n-1,\nu}^{-}),\pi_n S_{\nu}^{+})\|_{L(\ell_p(F_m^1),\ldots,\ell_p(F_m^{n-1}),\ell_p(G_m);Y^0+Y^1)} \\
&\leq \|T\|_{L(X_1^0\times \cdots \times X_{n-1}^0\times X_n^1;Y^0+Y^1)}
\|(\pi_1,\ldots,\pi_n)\|
\|R_{1,\nu}+R_{1,\nu}^{-}\| \cdots \|R_{n-1,\nu}+R_{n-1,\nu}^{-}\|
\|S_{\nu}^{+}\| \\
&\leq 2^{-\nu}\|T\|_{L(X_1^0\times \cdots \times X_{n-1}^0\times X_n^1;Y^0+Y^1)}
\to 0
\quad\, \text{as \, $\nu\to\infty$}\,.
\end{align*}

Set
\[
\|T_{8,\nu}\|_0:=\|T_{8,\nu}\|_{L(\ell_p(F_m^1),\ldots,\ell_p(F_m^{n-1}),\ell_p(G_m);\ell_\infty(W_m))}
\]
and
\[
\|T_{8,\nu}\|_1:=\|T_{8,\nu}\|_{L(\ell_p(2^{-m}F_m^1),\ldots,\ell_p(2^{-m}F_m^{n-1}),\ell_p(2^{-m}G_m);\ell_\infty(2^{-m}W_m))}\,.
\]

Applying Lemma \ref{lemma2} to the operator
\[
T(\pi_1(R_{1,\nu}+R_{1,\nu}^{-}),\ldots,\pi_{n-1}(R_{n-1,\nu}+R_{n-1,\nu}^{-}),\pi_n S_{\nu}^{+}),
\]
and using the boundedness of $Q$ and $P_{(n+1)\nu}^{+}$, we obtain the following. If $\widetilde{\beta}_0>0$, then there 
exist a constant $C_1>0$, independent of $T$, and a subsequence $(\nu')$ of $(\nu)_{\nu\ge1}$ such that
\[
\limsup_{\nu'\to\infty}\|T_{8,\nu'}\|_0 \le C_1\widetilde{\beta}_0\,.
\]
If $\widetilde{\beta}_0=0$, then, passing to a subsequence if necessary, we have
\[
\|T_{8,\nu'}\|_0 \to 0
\quad \text{as \, $\nu'\to\infty$}\,.
\]

Now let $u_i\in f_1^i$, $1\le i\le n-1$, and $v\in g_1$. Then $u_i\in f_0^i\cap f_1^i$ and $v\in g_0\cap g_1$. Hence
\begin{align*}
& \|T_{8,\nu'}(u_1,\ldots,u_{n-1},v)\|_{\ell_\infty(2^{-m}W_m)} \\
&\leq \|P_{(n+1)\nu'}^{+}\|_{L(\ell_\infty(W_m),\ell_\infty(2^{-m}W_m))}
\|\widehat{T}(u_1,\ldots,u_{n-1},v)\|_{\ell_\infty(W_m)} \\
&\leq 2^{-(n+1)\nu'}\|\widehat{T}(u_1,\ldots,u_{n-1},v)\|_{\ell_\infty(W_m)}
\to 0
\quad \text{as \, $\nu'\to\infty$}.
\end{align*}
Moreover,
\[
\|T_{8,\nu'}\|_1 \le \|P_{(n+1)\nu'}^{+}\widehat{T}\|_1.
\]
Therefore, if $\widetilde{\beta}_1>0$, then by Lemma \ref{lemma}, passing to a further subsequence if necessary, we have
\[
\limsup_{\nu'\to\infty}\|T_{8,\nu'}\|_1 \le C_2\widetilde{\beta}_1
\]
for some constant $C_2>0$ independent of $T$. If $\widetilde{\beta}_1=0$, then, passing to a subsequence if necessary, we have
\[
\|T_{8,\nu'}\|_1 \to 0 \quad\, \text{as \, $\nu'\to\infty$}.
\]

Passing to a common subsequence if necessary, again denoted by $(\nu')$, we obtain the two endpoint estimates above simultaneously.

If $\widetilde{\beta}_0=0$ or $\widetilde{\beta}_1=0$, then, arguing as in {\bf Step 8(a)}, we obtain
\[
\widetilde{\beta}\bigg(T_{8,\nu'} \colon \bigg(\prod_{i=1}^{n-1}E_i(F_m^i)\bigg)\times E_n(G_m)\to E(W_m)\bigg)\to 0
\quad \text{as \, $\nu'\to\infty$}\,.
\]

Assume now that $\widetilde{\beta}_0>0$ and $\widetilde{\beta}_1>0$. Fix $\varepsilon>0$. Then, for all sufficiently large $\nu'$,
\[
\|T_{8,\nu'}\|_0 \le (C_1+\varepsilon)\widetilde{\beta}_0\,,
\qquad
\|T_{8,\nu'}\|_1 \le (C_2+\varepsilon)\widetilde{\beta}_1\,.
\]
Hence, by Lemma \ref{lemma1} and Theorem \ref{conv},
\begin{align*}
& \widetilde{\beta}\bigg(T_{8,\nu'} \colon \bigg(\prod_{i=1}^{n-1}E_i(F_m^i)\bigg)\times E_n(G_m)\to E(W_m)\bigg) 
\leq \|T_{8,\nu'}\|_{L((\prod_{i=1}^{n-1}E_i(F_m^i))\times E_n(G_m);E(W_m))} \\
&\leq C (C_1+\varepsilon)\widetilde{\beta}_0
\varphi_{E_2}\bigg(\bigg(\frac{(C_2+\varepsilon)\widetilde{\beta}_1}{(C_1+\varepsilon)\widetilde{\beta}_0}\bigg)^{1/(n-1)}\bigg)\cdots
\varphi_{E_n}\bigg(\bigg(\frac{(C_2+\varepsilon)\widetilde{\beta}_1}{(C_1+\varepsilon)\widetilde{\beta}_0}\bigg)^{1/(n-1)}\bigg)\,.
\end{align*}
Using property $\ppp$ of the functions $\varphi_{E_i}$, $2\le i\le n$, we may absorb the constants $C_1+\varepsilon$ and $C_2+\varepsilon$ into $C$. Therefore,
\begin{align*}
\widetilde{\beta}\bigg(T_{8,\nu'} \colon \bigg(\prod_{i=1}^{n-1}E_i(F_m^i)\bigg) & \times E_n(G_m)\to E(W_m)\bigg) \\
& \le C\,\widetilde{\beta}_0
\varphi_{E_2}\bigg(\bigg(\frac{\widetilde{\beta}_1}{\widetilde{\beta}_0}\bigg)^{1/(n-1)}\bigg)\cdots
\varphi_{E_n}\bigg(\bigg(\frac{\widetilde{\beta}_1}{\widetilde{\beta}_0}\bigg)^{1/(n-1)}\bigg)
\end{align*}
for all sufficiently large $\nu'$. Consequently,
\begin{align*}
\limsup_{\nu'\to\infty}
\widetilde{\beta}\bigg(T_{8,\nu'} \colon \bigg(\prod_{i=1}^{n-1}E_i(F_m^i)\bigg) & \times E_n(G_m)\to E(W_m)\bigg) \\
& \le C\,\widetilde{\beta}_0
\varphi_{E_2}\bigg(\bigg(\frac{\widetilde{\beta}_1}{\widetilde{\beta}_0}\bigg)^{1/(n-1)}\bigg)\cdots
\varphi_{E_n}\bigg(\bigg(\frac{\widetilde{\beta}_1}{\widetilde{\beta}_0}\bigg)^{1/(n-1)}\bigg)\,.
\end{align*}

{\bf c)} We now consider the terms in which at least one factor is of the form $R_{\nu}+R_{\nu}^{-}$ and at least one factor 
is of the form $R_{\nu}^{+}$. For instance, let
\[
T_{9,\nu}:=P_{(n+1)\nu}^{+}\widehat{T}(R_{1,\nu}+R_{1,\nu}^{-},\ldots,R_{n-2,\nu}+R_{n-2,\nu}^{-},R_{n-1,\nu}^{+},I).
\]
The verification of the endpoint estimates is analogous to that in {\bf Step 8(b)}, but now the off-diagonal factor occurs 
in the $(n-1)$-st variable rather than in the last one. Repeating the argument of {\bf Step 8(b)}, we obtain the following. 
If $\widetilde{\beta}_0=0$ or $\widetilde{\beta}_1=0$, then, passing to a subsequence if necessary,
\[
\widetilde{\beta}\bigg(T_{9,\nu'} \colon \bigg(\prod_{i=1}^{n-1}E_i(F_m^i)\bigg)\times E_n(G_m)\to E(W_m)\bigg)\to 0
\quad \text{as \, $\nu'\to\infty$}.
\]
If $\widetilde{\beta}_0>0$ and $\widetilde{\beta}_1>0$, then, passing to a subsequence if necessary, we have
\begin{align*}
\limsup_{\nu'\to\infty}
\widetilde{\beta}\bigg(T_{9,\nu'} \colon \bigg(\prod_{i=1}^{n-1}E_i(F_m^i)\bigg) & \times E_n(G_m)\to E(W_m)\bigg) \\
& \le C\,\widetilde{\beta}_0
\varphi_{E_2}\bigg(\bigg(\frac{\widetilde{\beta}_1}{\widetilde{\beta}_0}\bigg)^{1/(n-1)}\bigg)\cdots
\varphi_{E_n}\bigg(\bigg(\frac{\widetilde{\beta}_1}{\widetilde{\beta}_0}\bigg)^{1/(n-1)}\bigg)\,.
\end{align*}
The same reasoning applies to every term containing at least one factor of the form $R_{\nu}+R_{\nu}^{-}$ 
and at least one factor of the form $R_{\nu}^{+}$.

{\bf d)} Finally, consider
\[
T_{10,\nu}:=P_{(n+1)\nu}^{+}\widehat{T}(R_{1,\nu}^{+}, \ldots, R_{n-1,\nu}^{+}, I)\,.
\]
This is the extreme case in which all the first $n-1$ variables are of off-diagonal type. Repeating the argument 
of {\bf Step 8(b)}, we obtain the following. If $\widetilde{\beta}_0=0$ or $\widetilde{\beta}_1=0$, then, passing to a subsequence if necessary,
\[
\widetilde{\beta}\bigg(T_{10,\nu'} \colon \bigg(\prod_{i=1}^{n-1}E_i(F_m^i)\bigg)\times E_n(G_m)\to E(W_m)\bigg)\to 0
\quad \text{as \, $\nu'\to\infty$}\,.
\]
If $\widetilde{\beta}_0>0$ and $\widetilde{\beta}_1>0$, then, passing to a subsequence if necessary, we have
\begin{align*}
\limsup_{\nu'\to\infty}
\widetilde{\beta}\bigg(T_{10,\nu'} \colon \bigg(\prod_{i=1}^{n-1}E_i(F_m^i)\bigg) & \times E_n(G_m)\to E(W_m)\bigg) \\
& \le C\,\widetilde{\beta}_0
\varphi_{E_2}\bigg(\bigg(\frac{\widetilde{\beta}_1}{\widetilde{\beta}_0}\bigg)^{1/(n-1)}\bigg)\cdots
\varphi_{E_n}\bigg(\bigg(\frac{\widetilde{\beta}_1}{\widetilde{\beta}_0}\bigg)^{1/(n-1)}\bigg)\,.
\end{align*}
Thus, the analysis of $P_{(n+1)\nu}^{+}\widehat{T}$ is complete.

\smallskip

{\bf Step 9}. For $P_{(n+1)\nu}^{-}\widehat{T}$, we argue exactly as in {\bf Step 8}, replacing $P_{(n+1)\nu}^{+}$ by $P_{(n+1)\nu}^{-}$ throughout. 
Consequently, the corresponding terms satisfy the same estimates.

\smallskip

{\bf Step 10}. Taking into account the decomposition \eqref{dec1} and the estimates obtained in the previous steps, we proceed as follows.

Assume first that $\widetilde{\beta}_j>0$, $j=0,1$. Since only finitely many terms appear in the decomposition \eqref{dec1}, by passing to a common 
subsequence if necessary, we may choose a subsequence $(\nu')$ such that all conclusions obtained in {\bf Steps 8} and {\bf 9} hold simultaneously 
along $(\nu')$. Moreover, by {\bf Steps 3--7}, all remaining terms tend to zero as $\nu'\to\infty$.

Combining these estimates, we obtain that for every $\varepsilon>0$, there exists $\nu'$ sufficiently large such that the sum of all terms arising in {\bf Steps 3--7} is bounded by 
$\varepsilon$, while the terms treated in {\bf Steps 8} and {\bf 9} satisfy the corresponding estimates obtained there. Combining these bounds with 
the estimate from {\bf Step 2}, we obtain
\begin{align*}
\widetilde{\beta}\big(T \colon (X_1^0,X_1^1)_{E_1} & \times \cdots \times (X_n^0,X_n^1)_{E_n} \to (Y^0,Y^1)_E\big) \\
& \leq C\,\widetilde{\beta}_0
\varphi_{E_2}\bigg(\bigg(\frac{\widetilde{\beta}_1}{\widetilde{\beta}_0}\bigg)^{1/(n-1)}\bigg)\cdots
\varphi_{E_n}\bigg(\bigg(\frac{\widetilde{\beta}_1}{\widetilde{\beta}_0}\bigg)^{1/(n-1)}\bigg)
+\varepsilon\,,
\end{align*}
where $C>0$ is independent of $T$.

Since $\varepsilon>0$ is arbitrary, it follows that
\begin{align*}
\widetilde{\beta}\big(T \colon (X_1^0,X_1^1)_{E_1} & \times \cdots \times (X_n^0,X_n^1)_{E_n} \to (Y^0,Y^1)_E\big) \\
& \leq C\,\widetilde{\beta}_0
\varphi_{E_2}\bigg(\bigg(\frac{\widetilde{\beta}_1}{\widetilde{\beta}_0}\bigg)^{1/(n-1)}\bigg)\cdots
\varphi_{E_n}\bigg(\bigg(\frac{\widetilde{\beta}_1}{\widetilde{\beta}_0}\bigg)^{1/(n-1)}\bigg)\,.
\end{align*}

Since, for each $2\leq i\leq n$, we have $J_{E_i}(\xo_i)=K_{E_i}(\xo_i)$ up to equivalence of norms, this is precisely the required estimate.

If $\widetilde{\beta}_j=0$ for either $j=0$ or $j=1$, then {\bf Steps 8} and {\bf 9} yield zero contribution along a suitable subsequence, 
while {\bf Steps 3--7} still tend to zero. Therefore,
\[
\widetilde{\beta}\big(T \colon (X_1^0,X_1^1)_{E_1} \times \cdots \times (X_n^0,X_n^1)_{E_n} \to (Y^0,Y^1)_E\big)=0.
\]
This completes the proof.
\end{proof}

As a consequence of the main theorem, we obtain the following one-sided compactness result for interpolated multilinear operators.

\begin{corollary}
Under the assumptions of Theorem {\rm\ref{main}}, assume that, for at least one $j \in \{0,1\}$, the operator
is compact, then the restriction of $T$ is compact from $(X_1^0,X_1^1)_{E_1} \times \cdots \times (X_n^0,X_n^1)_{E_n}$ 
to $(Y^0,Y^1)_E$. 
\end{corollary}

In particular, under the above compactness assumption, in the Banach setting, if $1 \leq p_1,\ldots,p_n,q \leq \infty$ satisfy
\[
0 \leq \frac1q \leq \frac1{p_1}+\cdots+\frac1{p_n}-(n-1)\,,
\]
then, for every $\theta \in (0,1)$, the restriction of $T$ is compact from $(X_1^0,X_1^1)_{\theta,p_1} \times \cdots \times (X_n^0,X_n^1)_{\theta,p_n}$
to $(Y^0,Y^1)_{\theta,q}$.

\vspace{0.3cm}

\noindent Mieczys{\l}aw Masty{\l}o \\
Faculty of Mathematics and Computer Science\\
Adam Mickiewicz University, Pozna\'n\\
61-614 Pozna{\'n}, Poland

\vspace{1 mm}

\noindent
E-mail: \,{\tt mieczyslaw.mastylo$@$amu.edu.pl} \\

\noindent
Eduardo Brandani da Silva  \\
Departamento de Matem\'atica \\
Universidade Estadual de Maring\'a--UEM \\
Av.~Colombo 5790 \\
Maring\'a - PR  \\
Brazil - 870300-110

\vspace{1 mm}

\noindent E-mail: \,{\tt ebsilva@uem.br}

\bigskip


\begin{thebibliography}{99}

\bibitem{benyi1}
\'A.~Bényi and T.~Oh, \emph{Smoothing of commutators for a Hörmander class of bilinear pseudodifferential operators}, 
J. Fourier Anal. Appl. \textbf{20}~(2014), 282--300.

\bibitem{benyi2}
\'A. Bényi, R.\,H.~Torres, \emph{Compact bilinear operators and commutators}, Proc. Amer. Math. Soc. \textbf{141}~(2013), 
3609--3621.

\bibitem{BL}
J.~Bergh and J.~L\"ofstr\"om, \emph{Interpolation Spaces. An Introduction}, Springer, Berlin 1976.

\bibitem{BC}
B.\,F.~Besoy and F.~Cobos, \emph{Interpolation of the measure of non-compactness of bilinear operators among quasi-Banach 
spaces}, J. Approx. Theory \textbf{243}~(2019), 25--44.

\bibitem{BK}
Y.~Brudnyi and N.~Kruglyak, \emph{Interpolation functors and interpolation spaces}, Volume 1, 
North-Holland, Amsterdam 1991.

\bibitem{Cal}
A.\,P.~Calder\'on, \emph{Intermediate spaces and interpolation, the complex method},
Studia Math. \textbf{24}~(1964), 113--190.

\bibitem{CFCK} 
F.~Cobos, L.\,M.~Fern\'andez-Cabrera and T.~K\"uhn, \emph{Interpolation of compact multilinear operators between quasi-Banach spaces}, 
J. Approx. Theory \textbf{314}~(2026), part 1, Paper No. 106222, 20 pp.


\bibitem{CLM1}
F.~Cobos, L.\,M.~Fern\'andez-Cabrera and A.~Martínez, \emph{Interpolation of compact bilinear operators among 
quasi-Banach spaces and applications}, Math. Nachr. \textbf{291}~(2018), 2168--2187.


\bibitem{EE}
D.\,E.~Edmunds and W.\,D.~Evans, \emph{Spectral Theory and Differential Operators}. Oxford Mathematical Monographs.
Oxford Science Publications. The Clarendon Press, Oxford University Press, New York, 1987.

\bibitem{KPR} 
N.\,J.~Kalton, N.\,T.~Peck and J.\,W.~Roberts, \emph{An $F$-space Sampler}, Cambridge 1984. 

\bibitem{LP}
J.-L.~Lions and J.~Peetre, \emph{Sur une classe d'espaces d'interpolation}, Inst. Hautes \'Etudes Sci. Publ. Math.
\textbf{19}~(1964), 5--68


\bibitem{MS}
M.~Masty{\l}o and E.\,B.~Silva, \emph{Interpolation of the measure of noncompactness of bilinear operators},
Trans. Amer. Math. Soc. \textbf{370}~(2018), no.~12, 8979--8997.

\bibitem{Oberlin}
D.~M.~Oberlin, \emph{A multilinear Young's inequality}, Canad. Math. Bull. \textbf{31}~(1988), no.~3, 380--384.

\bibitem{Rolewicz} 
S.~Rolewicz, \emph{Metric Linear Spaces}. PWN, Warszawa, 1972; 2-ed Ed., Warszawa 1984. 


\bibitem{Szw}
R.~Szwedek, \emph{Measure of non-compactness of operators interpolated by the real method},
Studia Math. \textbf{175}~(2006), no.~2, 157--174.


\bibitem{Zafran}
M.~Zafran, \emph{A multilinear interpolation theorem}, Studia Math. \textbf{62}~(1978), 107--124.
\end{thebibliography}
\end{document}